\documentclass[11pt,letterpaper]{article}
\pdfoutput=1
\usepackage{fullpage}
\usepackage{graphicx,psfrag}
\usepackage[noend]{algorithmic}
\usepackage{xspace}
\usepackage{amssymb,amsmath,amscd,mathrsfs}
\usepackage{amsthm}
\usepackage{mdframed}
\usepackage{comment,color}
\usepackage{subcaption}
\usepackage[font=scriptsize]{caption}
\usepackage{url}
\usepackage{multirow}
\usepackage{enumerate}
\usepackage{tikz}
\usepackage{cuted}
\usetikzlibrary{shapes.gates.logic.US,trees,positioning,arrows}
\usepackage{mdframed}
\usepackage[title]{appendix}

\usepackage[bookmarks=true]{hyperref}

\surroundwithmdframed[innerleftmargin=3pt,innerrightmargin=3pt,innerbottommargin=3pt]{quote}

\usepackage{color}
\usepackage{amsmath}
\usepackage{amssymb}
\usepackage{graphicx}
\usepackage{comment,xspace}
\usepackage{fancybox}

\newcommand{\real}{{\mathbb{R}}}
\newcommand{\reals}{\real}

\renewcommand{\natural}{{\mathbb{N}}}
\newcommand{\naturals}{\natural}

\newtheorem{theorem}{Theorem}[section]

\newtheorem{lemma}[theorem]{Lemma}
\newtheorem{remark}[theorem]{Remark}
\newtheorem{example}[theorem]{Example}
\newtheorem{definition}[theorem]{Definition}
\newtheorem{corollary}[theorem]{Corollary}

\newcommand{\probnoarg}{\mbox{$\mathbb{P}$}} 
\newcommand{\expectation}[1]{\mbox{$\mathbb{E}\left[#1\right]$}}

\newtoggle{EV}
\toggletrue{EV}

\newcommand{\revision}[1]{{\color{black}{#1}}}
\newcommand{\revisionII}[1]{{\color{black}{#1}}}

\patchcmd{\quote}{\rightmargin}{\leftmargin 0.1em \rightmargin}{}{}

\usepackage{epstopdf}

\newcommand{\emax}{\mathcal{E}_{\max}}

\newcommand{\fil}{\mathcal F}
\newcommand{\cs}{\mathcal Z}
\newcommand{\csd}{\mathcal U}

\newcommand{\risk}{\rho}

\newcommand{\upol}{\mathcal U^{\mathrm{poly}}}
\newcommand{\upolv}{\mathcal U^{\mathrm{poly}, V}}

\title{\LARGE \bf
A Framework for Time-Consistent, Risk-Sensitive \\Model Predictive Control: Theory and Algorithms \footnotetext{\llap{\textsuperscript{ }}This work was supported by ONR under the Science of Autonomy Program, Contract N00014-15-1-2673.}}

\author{ Sumeet Singh\thanks{Department of Aeronautics and Astronautics, Stanford University, CA 94305, USA. Emails: {\tt \{ssingh19, pavone\}@stanford.edu}.} \and
Yin-Lam Chow\thanks{Institute for Computational and Mathematical Engineering, Stanford University, Stanford, CA 94305, USA. Email: {\tt ychow@stanford.edu}.} \and 
Anirudha Majumdar\thanks{Department of Mechanical and Aerospace Engineering, Princeton University, NJ 08544, USA. Email: {\tt ani.majumdar@princeton.edu}.} 	\and
Marco Pavone\footnotemark[1]
}

\date{}

\begin{document}
\maketitle
\thispagestyle{plain}
\pagestyle{plain}

\maketitle
\begin{abstract}
In this paper we present a framework for risk-sensitive model predictive control (MPC) of linear systems affected by \revisionII{stochastic} multiplicative uncertainty. Our key innovation is to consider a time-consistent, dynamic risk evaluation \revision{of the cumulative cost} as the objective function to be minimized. This framework is axiomatically justified in terms of time-consistency of risk assessments, is amenable to dynamic optimization, and is unifying  in the sense that it captures a full range of risk preferences from risk-neutral \revisionII{(i.e., expectation)} to worst case. Within this framework, we propose and analyze an online risk-sensitive MPC algorithm that is provably stabilizing. Furthermore, by exploiting the dual representation of time-consistent, dynamic risk measures, we cast the computation of the MPC control law as a convex optimization problem amenable to real-time implementation. Simulation results are presented and discussed.
\end{abstract}

\section{Introduction}\label{sec:intro}

\revisionII{Safety-critical control demands the consideration of \revision{uncertain} events, and in particular of events with small probabilities that can nevertheless have catastrophic effects if realized. Accordingly, one of the \revision{current} main research thrusts \revision{in} Model Predictive Control (MPC) is to \revision{find techniques that can robustly address uncertainty~ \cite{Mayne2014,KouvaritakisCannon2015,Mesbah2016}.}} Techniques for handling uncertainty within the MPC framework fall \revision{broadly} into \revision{three categories: (1) min-max formulations, where the performance indices to be minimized are computed with respect to the worst possible disturbance realization\iftoggle{EV}{~\cite{KothareBalakrishnanEtAl1996,Loefberg2003,Mayne2014}}{}, (2) tube-based formulations, where classical (uncertainty-unaware) MPC is modified to use tightened constraints and augmented with a tracking ancillary controller to maintain the system within an invariant tube around the nominal MPC trajectory\iftoggle{EV}{~\cite{MayneSeronEtAl2005,RakovicKouvaritakisEtAl2012,CannonBuergerEtAl2011}}{}, and (3) stochastic formulations, where \emph{risk-neutral expected} values of performance indices (and possibly constraints) are considered\iftoggle{EV}{~\cite{PrimbsSung2009,BernardiniBemporad2012, CannonKouvaritakisEtAl2011, FlemingCannonEtAl2014, Mayne2016} (see also the recent reviews \cite{Mayne2014,Mesbah2016})}{~\cite{Mayne2014,KouvaritakisCannon2015,Mesbah2016}}. The main drawback of the min-max approach is that the control law may be too conservative, since the performance index is being optimized under the worst-case disturbance realizations (which may have an arbitrarily small probability of occurring). The tightened constraints in tube-based formulations induce similar conservatism upon the optimized cost function. On the other hand, stochastic formulations, where the assessment of future random outcomes is accomplished through the expectation operator, may be unsuitable in scenarios where one desires to account for {\em risk}, i.e., increased awareness of events of small probability and detrimental consequences.}

In general, there are three main challenges with incorporating risk-sensitivity into control and decision-making problems:

{\bf{Rationality and consistency:}} The behavior of a control system using a certain risk measure (i.e., a function that maps an uncertain cost to a real number) should be consistent over time. Intuitively, time-consistency stipulates that if a given sequence of costs incurred by the system, when compared to another sequence, has the same current cost and lower risk in the future, then it should be considered less risky at the current time (see Section \ref{subsection_time_cons} for a formal statement). Examples of ``irrational" behavior that can result from a time-inconsistent risk measure include: (1) a control system intentionally seeking to incur losses \cite{MannorTsitsiklis2011}, or (2) deeming states to be dangerous when in fact they are favorable under any realization of the underlying uncertainty \cite{Sarlette2009}, or (3) declaring a decision-making problem to be feasible (e.g., satisfying a certain risk threshold) when in fact it is infeasible under \emph{any} possible subsequent realization of the uncertainties \cite{RoordaSchumacherEtAl2005}. Remarkably, some of the most common strategies for incorporating risk aversion in decision-making (discussed below) display such inconsistencies \cite{Ruszczynski2010,Sarlette2009}.

{\bf{Computational tractability:}} A risk measure generally adds a nonlinear structure to the optimization problem one must solve in order to compute optimal actions. Hence, it is important to ensure the computational tractability of the optimization problem induced by the choice of \revision{a} risk measure, particularly in dynamic decision-making settings where the control system must plan and react to disturbances in real-time.

{\bf{Modeling flexibility:}} One would like to calibrate the risk measure to the control application at hand by: (1) exploring the full spectrum of risk assessments from worst-case to risk-neutral, and (2) ensuring that the risk measure can be applied to a rich set of uncertainty models (e.g., beyond Gaussian models).

Most popular methods in the literature for assessing risks do not satisfy these three requirements. The Markowitz mean-variance criterion~\cite{Markowitz1952}, which has dominated risk management for over $50$ years, leads to time-inconsistent assessments of risk in a multi-stage stochastic control \revision{setting} and also \revision{generally leads to} computationally intractable problems \cite{MannorTsitsiklis2011}.  Moreover, it is rather limited in terms of modeling flexibility \revision{since it relies only on the first two moments of the distribution} and there is only a single tuning parameter \revision{to} trade off between these two moments. Thus, the mean-variance criterion is not well-suited to applications where the disturbance model is non-Gaussian \revision{and has been shown to drastically underestimate the effect of extreme events characterized by severe losses~\cite{Szegoe2005}.}

A popular alternative to the mean-variance criterion is the entropic risk measure: $\rho(X) = \log\left(\mathbb{E}[e^{\theta X}]\right)/\theta$, \revision{with} $\theta\in(0,1)$. The entropic risk measure has been widely studied in the financial mathematics\iftoggle{EV}{~\cite{GerberPafum1998,FlemingSheu2000}}{~\cite{FlemingSheu2000}} and sequential decision making~\cite{HowardMatheson1972, MihatschNeuneier2002,BaeuerleRieder2013} literatures, and for modeling risk aversion in LQG control problems \cite{Whittle1981,GloverDoyle1987}. While the entropic risk is a more computationally tractable alternative to the mean-variance criterion and can also lead to time-consistent behavior~\cite{DetlefsenScandolo2005}, practical applications of the entropic risk measure have proven to be problematic. Notice that the first two terms of the Taylor series expansion of $\rho(X)$ form a weighted sum of mean and variance with regularizer $\theta$, i.e., $\rho(X)\approx\mathbb E(X)+ (\theta/2) \mathbb{E}(X-\mathbb E[X])^2$. Consequently, the primary concerns are similar to those associated with the mean-variance measure of risk, \revision{with the added complication that the exponential term may induce numerical conditioning problems~\cite{Whittle2002}.}
The entropic risk measure is a particular example of the general class of methods that model risk aversion \revision{by} using concave utility functions (convex disutility functions in the cost minimization setting). While the expected (dis)utility framework captures the intuitive notion of diminishing marginal utility, it suffers from the issue that even very little risk aversion over moderate costs leads to unrealistically high degrees of risk aversion over large costs \cite{Rabin2000,WangWebsterEtAl2009} (note that this is a limitation of \emph{any} concave utility function). 
\revision{Additionally, the linear treatment of stochasticity in expected (dis)utility theory is \revisionII{in general} too restrictive to be able to \revisionII{account for} more general statistics, particularly \revisionII{in the context of capturing extreme} events~\cite[Chapter 6]{Aven2009}. \revisionII{Specifically,} within the expected (dis)utility model, such events would be simply averaged over along with the low cost (or ``safe'') events. We refer the reader to~\cite[Chapter 6]{Aven2009} where the limitations of purely using utility functions for representing risk are discussed from a \revisionII{foundational perspective}}.


In order to overcome such challenges, in this paper we incorporate risk sensitivity in MPC by leveraging recent strides in the theory of \emph{dynamic} risk measures developed by the operations research community \cite{Ruszczynski2010}. This allows us to propose a framework that satisfies the requirements outlined above with respect to rationality and consistency, computational tractability, and modeling flexibility. Specifically, the key property of dynamic risk measures is that, by reassessing risk at multiple points in time, one can guarantee time-consistency of risk preferences and the agent's behavior \cite{Ruszczynski2010}. In particular, it is proven in \cite{Ruszczynski2010} that time-consistent risk measures can be represented as a \emph{composition} of {\em one-step} \emph{coherent risk measures}. Coherent risk measures~\cite{ArtznerDelbaenEtAl1999,RuszczynskiShapiro2006} have been thoroughly investigated and widely applied for static decision-making problems in operations research and finance. Coherent risk measures were originally conceived in  \cite{ArtznerDelbaenEtAl1999} from an \emph{axiomatization} of properties that any rational agent's risk preferences should satisfy (see Section \ref{sec:prelim} for a formal statement of these axioms). In addition to being axiomatically justified, coherent risk measures capture a wide spectrum of risk assessments from risk neutral to worst-case and thus provide a unifying approach to static risk assessments. Since time-consistent dynamic risks are composed of one-step coherent risks, they inherit the same modeling flexibility.

%

\revision{{\em Statement of Contributions}}: The contribution of this paper is threefold. First, we introduce a class of dynamic risk measures, referred to as Markov dynamic polytopic risk measures, that capture a full range of risk assessments and enjoy a geometrical structure that is particularly favorable from a computational standpoint. Second, we present and analyze a \emph{risk-sensitive} MPC algorithm that minimizes in a receding-horizon fashion a Markov dynamic polytopic risk measure, under the assumption that the system's model is linear and is affected by \revisionII{stochastic} multiplicative uncertainty. Finally, by exploring the \emph{geometric} structure of Markov dynamic polytopic risk measures, we present a convex programming formulation for risk-sensitive MPC that is amenable to a real-time implementation (for moderate horizon lengths). Our framework has three main advantages: (1) it is axiomatically justified, in the sense that risk, by construction, is assessed in a time-consistent fashion; (2) it is amenable to dynamic and convex optimization, primarily due to the compositional form of Markov dynamic polytopic risk measures and their  geometry; and (3) it is general, in that it captures a full range of risk assessments from risk-neutral to worst-case. In this respect, our formulation represents a \emph{unifying} approach for risk-sensitive MPC. 

\revision{Our approach is inspired by the work in~\cite{Ruszczynski2010}, whereby the authors consider a similar risk-sensitive objective function for controlled Markov processes within an infinite-horizon formulation. Differently from \cite{Ruszczynski2010}, in this paper we consider an MPC formulation, define and address notions of persistent feasibility and stability, and provide real-time algorithms for the solution of problems with continuous state- and control-spaces.}

A preliminary version of this paper was presented in~\cite{ChowPavone2014}. In this extended and revised version, we present the following key extensions: (1) the introduction of constraints on state and control variables, (2) a new offline/online MPC formulation for handling these constraints, and (3) additional numerical experimental results including \revision{(i) an illustration of the effects of varying levels of risk-sensitivity (thereby also providing a comparison with the standard risk-neutral MPC formulation), and (ii) a scalability study to assess the computational limitations of the proposed approach}.

\revision{\em {Organization}}: The rest of the paper is organized as follows. In Section \ref{sec:prelim} we provide a review of the theory of dynamic risk measures. In Section \ref{sec:sys} we discuss the stochastic model we address in this paper. In Section \ref{sec:risk} we introduce and discuss the notion of Markov dynamic polytopic risk measures. In Section \ref{sec:IHC} we state the infinite horizon optimal control problem we wish to address and in Section \ref{sec:stab} we derive conditions for risk-sensitive closed-loop stability.  In Section \ref{sec:MPC} we present the MPC adaptation of the infinite horizon problem and present various \revision{solution} approaches  in Section \ref{sec:alg}. Numerical experiments are presented and discussed in Section \ref{sec:example}. Finally,  in Section \ref{sec:conclusion} we draw some conclusions and discuss directions for future work. \iftoggle{EV}{}{\revision{Due to space limitations, all proofs \revisionII{of the theoretical results} are provided in \revision{an extended version of this paper}~\cite{SinghChowEtAl2018}.}}

\section{Review of Dynamic Risk Theory}\label{sec:prelim}
In this section, we briefly review the theory of coherent and dynamic risk measures, on which we will rely extensively in this paper. The material presented in this section summarizes several novel results in risk theory achieved in the past ten years. Our presentation strives to present this material in an intuitive fashion and with a notation tailored to control applications.
\subsection{Static, Coherent Measures of Risk}\label{subsec:static_risk}
Consider a probability space $(\Omega, \fil, \probnoarg)$, where $\Omega$ is the set of outcomes (sample space), $\fil$ is a $\sigma$-algebra over $\Omega$ representing the set of events we are interested in, and $\probnoarg$ is a probability measure over $\fil$. In this paper we will focus on disturbance models characterized by probability \emph{mass} functions (pmfs), hence we restrict our attention to finite probability spaces (i.e.,  $\Omega$ has a finite number of elements or, equivalently, $\fil$ is a finitely generated algebra). 
Denote with $\cs$ the space of random variables $Z:\Omega\mapsto (-\infty, \infty)$ defined over the probability space $(\Omega, \fil, \mathbb P)$. In this paper a random variable $Z\in \cs$ is interpreted as a cost, i.e., the smaller the realization of $Z$, the better. 
For $Z, W$, we denote by $Z\leq W$ the point-wise partial order, i.e., $Z(\omega)\leq W(\omega)$ for all $\omega\in \Omega$.

By a \emph{risk measure} we understand a function $\risk(Z)$ that maps an uncertain outcome $Z$ into the extended real line $\reals \cup\{ +\infty\}\cup \{-\infty\}$. In this paper we restrict our analysis to \emph{coherent risk measures}, defined as follows:

\begin{definition}[Coherent Risk Measures]\label{def:crm}
A coherent risk measure is a mapping $\risk:\cs \rightarrow \reals$, satisfying the following four axioms: for all $Z,W \in \cs$,
\begin{itemize}
\item[A1] Monotonicity:  $Z\leq W \Rightarrow \risk(Z)\leq\risk(W)$;
\item[A2] Translation invariance: $\forall a\in \reals$, $\risk(Z+a)=\rho(Z) + a$;
\item[A3] Positive homogeneity: $\forall \lambda\geq0$, $\risk(\lambda Z) = \lambda \risk(Z)$;
\item[A4] Subadditivity: $\risk(Z+W) \leq \risk(Z) + \risk(W)$.
\end{itemize}
\end{definition} 
These axioms were originally conceived in \cite{ArtznerDelbaenEtAl1999} and ensure the ``rationality" of one-step risk assessments (we refer the reader to \cite{ArtznerDelbaenEtAl1999} for a detailed motivation of these axioms). One of the main properties for coherent risk measures is a universal representation theorem for coherent risk measures, which in the context of \emph{finite} probability spaces takes the following form:

\begin{theorem}[Representation Theorem for Finite Probability Spaces \cite{ArtznerDelbaenEtAl1999}]\label{thrm:rep_finite}
Consider the probability space $\{\Omega, \fil, \probnoarg\}$ where $\Omega$ is finite with cardinality $L \in \naturals$, $\fil = 2^{\Omega}$, and $\probnoarg = (p(1), \ldots, p(L))$, with all probabilities positive.  Let $\mathcal B$ be the set of probability density functions:
$\mathcal B:=\Bigl \{ \zeta\in \reals^L: \sum_{j=1}^L \, p(j)\zeta(j)=1, \zeta\geq 0 \Bigr\}$.
The risk measure $\risk:\cs \rightarrow \reals$ is a coherent risk measure if and only if there exists a convex bounded and closed set $\csd \subset \mathcal B$ such that $
\risk(Z)=\max_{\zeta\in \csd} \mathbb E_{\zeta}[Z]$.
\end{theorem}

This result states that any coherent risk measure \revision{can be written as} an expectation with respect to a worst-case density function $\zeta$, chosen adversarially from a suitable \emph{set} of test density functions  (referred to as the \emph{risk envelope}). 


%


\subsection{Dynamic, Time-Consistent Measures of Risk}\label{subsection_time_cons}
This section provides a multi-period generalization of the concepts presented in Section \ref{subsec:static_risk} and follows closely the discussion in \cite{Ruszczynski2010}. Consider a probability space $(\Omega, \fil, \mathbb P)$, a filtration $\fil_0\subset \fil_1\subset \fil_2 \cdots \subset \fil_N \subset \fil$, and an adapted sequence of real-valued random variables $Z_k$, $k\in \{0, \ldots,N\}$. We assume that $\fil_0 = \{\Omega, \emptyset\}$, i.e., $Z_0$ is deterministic. The variables $Z_k$ can be interpreted as stage-wise costs. For each $k\in\{0, \ldots, N\}$, denote with $\cs_k$ the space of random variables defined over the probability space $(\Omega, \fil_k, \mathbb P)$; also, let $\cs_{k, N}:=\cs_k \times \cdots \times \cs_N$. Given sequences $Z = \{Z_k,\ldots, Z_N\}\in \cs_{k, N}$ and $W=\{W_k,\ldots, W_N\}\in \cs_{k, N}$, we interpret $Z\leq W$ component-wise, i.e., $Z_j\leq W_j$ for all $j\in \{k,\ldots, N\}$.

A dynamic risk measure is a sequence of mappings  $\risk_{k,N}:\cs_{k, N}\rightarrow\cs_k$, $k\in\{0, \ldots,N\}$, obeying the following monotonicity property:
$\risk_{k,N}(Z)\!\leq \!\risk_{k,N}(W)   \text{ for all } Z,W \!\in\!\cs_{k,N}$ such that $Z\leq W$. This monotonicity property (analogous to axiom A2 in Definition \ref{def:crm}) is a natural requirement for any meaningful dynamic risk measure. In this paper, we restrict our discussion to dynamic risk measures that ensure time-consistency~\cite{Ruszczynski2010}. Informally, this property states that if a certain `situation' is considered less risky than another situation in all states of the world at stage $k+1$, then it should also be considered less risky at stage $k$. \revision{Arguably, this is a desirable property to enforce when {\em designing} controllers for automatic control systems (in contrast to, e.g., {\em analyzing} human's behavior, which may or may not display such a property).}

To define the functional form of time-consistent dynamic risk measures, one must first generalize static coherent risk measures as follows:

\begin{definition}[Coherent One-step Conditional Risk Measures (\cite{Ruszczynski2010})]\label{def:coh}
A coherent one-step conditional risk measure is a mapping $\risk_k:\cs_{k+1}\rightarrow \cs_k$, $k\in\{0,\ldots,N-1\}$, with the following four properties: for all $W,W' \in \cs_{k+1}$  
\begin{itemize}
\item Monotonicity:  $W \leq W' \Rightarrow \risk_k(W)\leq\risk_k(W')$;
\item Translation invariance:  $\forall  Z \in \cs_k$, $\risk_k(Z+W)= Z + \risk_k(W)$;
\item Positive homogeneity: $\forall \lambda \geq 0$, $\risk_k(\lambda W) = \lambda \risk_k(W)$;
\item Subadditivity: $\risk_k(W+W') \leq \risk_k(W) + \risk_k(W')$.
\end{itemize}
\end{definition} 

%
We now state the main result of this section.
\begin{theorem}[Dynamic, Time-consistent Risk Measures (\cite{Ruszczynski2010})]\label{thrm:tcc}
Consider, for each $k\in\{0,\ldots,N\}$, the mappings $\risk_{k,N}:\cs_{k, N}\rightarrow\cs_k$ defined as
\begin{equation}\label{eq:tcrisk}
\begin{split}
\risk_{k,N} = Z_k &+ \risk_k(Z_{k+1} + \risk_{k+1}(Z_{k+2}+\ldots \\
					&+\risk_{N-2}(Z_{N-1}+\risk_{N-1}(Z_N))\ldots)),
\end{split}
\end{equation}
where the $\risk_k$'s are coherent one-step conditional risk measures. Then, the ensemble of such mappings is a dynamic, time-consistent  risk measure.
\end{theorem}

Remarkably, Theorem 1 in \cite{Ruszczynski2010} shows (under weak assumptions) that the ``multi-stage composition" in equation \eqref{eq:tcrisk} is indeed \emph{necessary for time-consistency}. Accordingly, in the remainder of this paper, we will focus on the \emph{dynamic, time-consistent risk measures} characterized in Theorem \ref{thrm:tcc}.

\section{Model Description}\label{sec:sys}
Consider the discrete time system:
\begin{equation}
x_{k+1}=A(w_k)x_k+B(w_k)u_k,\label{eqn_sys}
\end{equation}
where $k\in \naturals$ is the time index, $x_k\in\reals^{N_x}$ is the state, $u_k\in\reals^{N_u}$ is the (unconstrained) control input, and $w_k\in\mathcal{W}$ is the process disturbance. We assume that the initial condition $x_0$ is deterministic and that $\mathcal W$ is a finite set of cardinality $L$, i.e., $\mathcal W = \{w^{[1]}, \ldots, w^{[L]}\}$. Accordingly, denote $A_{j}:=A(w^{[j]})$ and $B_{j} := B(w^{[j]})$, $j\in \{1, \ldots, L\}$.

For each stage $k$  and state-control pair $(x_k, u_k)$, the process disturbance $w_k$ is drawn from set $\mathcal W$ according to the pmf $p=[p(1),\, p(2),\ldots,\, p(L)]^\top $,
where $p(j)=\mathbb{P}(w_k=w^{[j]})$, $j\in\{1,\ldots,L\}$. Without loss of generality, we assume that $p(j)>0$ for all $j$. Note that the pmf for the process disturbance is time-invariant, and that the process disturbance is \emph{independent} of the process history and of the state-control pair $(x_k, u_k)$. Under these assumptions, the stochastic process $\{x_k\}$ is clearly a Markov process. 

\revision{
\begin{remark}\label{rem:TV_p}
The results presented in this paper can be immediately extended to the case where the process disturbance pmf is time-varying (for example, it is driven by a separate stationary Markov process, as in the popular Markov Jump model~\cite{SouzaC.2006}) by defining an augmented state such as $(x_k,w_{k-1})$. We omit this generalization  in the interest of brevity and clarity.
\end{remark}
}

\section{Markov Polytopic Risk Measures}\label{sec:risk}

In this section we \emph{refine} the notion of dynamic, time-consistent risk measures (as defined in Theorem \ref{thrm:tcc}) in two ways: (1) we add a polytopic structure to the dual representation of coherent risk measures, and (2) we add a Markovian structure. This will lead to the definition of Markov dynamic polytopic risk measures, which enjoy favorable computational properties and, at the same time,  maintain most of the generality of dynamic, time-consistent risk measures.

\subsection{Polytopic Risk Measures}
According to the discussion in Section \ref{sec:sys}, the probability space for the process disturbance has a finite number of elements. \revision{Thus}, one has $\mathbb E_{\zeta}[Z]  =\sum_{j=1}^L\, Z(j) p(j) \zeta(j)$. In our framework (inspired by \cite{EichhornRoemisch2005}), we consider coherent risk measures where the risk envelope $\csd$ is a \emph{polytope}, i.e., there exist matrices $S^I$, $S^E$ and vectors $T^I$, $T^E$ of appropriate dimensions such that 
\begin{equation*}\label{polytope_set_dual}
\upol=\left\{\zeta \in \mathcal B  \mid S^I \, \zeta \leq T^I,\,\, S^E \zeta= T^E \right\}.
\end{equation*}
We will refer to coherent risk measures representable with a polytopic risk envelope as \emph{polytopic risk measures}. Consider the bijective map $q(j):=p(j) \zeta(j)$ (recall that, in our model, $p(j)>0$). Then, by applying such a map, one can easily rewrite a polytopic  risk measure as 
\[
\risk(Z)=\max_{q\in \upol} \mathbb E_{q}[Z],
\]
where $q$ is a \emph{pmf} belonging to a polytopic subset of the standard simplex, i.e.:
\begin{equation}\label{eq:rep_fin}
\upol= \Bigl \{ q\in \Delta^L \mid  S^I  q \leq T^I,\,\, S^E q= T^E  \Bigr  \},
\end{equation}
where $\Delta^L:=\bigl \{ q\in \reals^L: \sum_{j=1}^L \, q(j)=1, q \geq 0 \bigr\}$. Accordingly, one has $E_{q}[Z] = \sum_{j=1}^L \, Z(j) q(j)$ (note that, with a slight abuse of notation, we are using the same symbols as before for $\upol$, $S^I$, and $S^E$). \revision{We will refer to the set of vertices of $\upol$ as $\upolv$.}

The class of polytopic risk measures is \revision{large}; common examples include the expected value (the polytope reduces to the singleton pmf $\{p\}$) and the Conditional Value-at-Risk ($\text{CVaR}_{\alpha}$), defined as:
\begin{equation}
\text{CVaR}_{\alpha}(Z):=\inf_{y\in\reals}\left[y+\frac{1}{\alpha}\expectation{(Z-y)^+}\right], \label{CVaR}
\end{equation}
where $\alpha\in(0,1]$. \revision{Its} polytopic risk envelope is given by:
\[
\upol =\Bigl \{q\in \Delta^L  \mid 0\leq q(j) \leq \frac{p(j)}{\alpha},  j\in \{1,\ldots, L\}  \Bigr\}.
\]
Additional examples include semi-deviation measures, comonotonic risk measures
, spectral risk
, and optimized certainty equivalent
; see~\cite{MajumdarPavone2017} for further examples. The key point is that polytopic risk measures  \emph{cover a full gamut of risk assessments}, ranging from risk-neutral to worst case.

\subsection{Markov Dynamic Polytopic Risk Metrics}

Note  that in the definition of dynamic, time-consistent risk measures, since at stage $k$ the value of $\risk_k$ is $\fil_k$-measurable, the evaluation of risk can depend on the \emph{whole} past, see \cite[Section IV]{Ruszczynski2010}. For example, \revision{the level $\alpha$ in the definition of the $\mathrm{CVaR}_{\alpha}$ risk measure can be an $\fil_k$-measurable random variable (see \cite[Example 3]{Ruszczynski2010})}. This generality, which appears of little practical value in many cases, leads to optimization problems that are intractable. This motivates us to add a \emph{Markovian structure} to dynamic, time-consistent risk measures (similarly as in \cite{Ruszczynski2010}). In particular, we consider dynamic risk measures $\rho_{k,N}$ as defined in eq.~\eqref{eq:tcrisk} with coherent one-step conditional risk measures $\rho_k(\cdot)$ of the form:
\begin{equation}
\risk_k(Z(x_{k+1}))=\max_{q\in\upol_{k}(x_k,p)}\mathbb{E}_{q}[Z(x_{k+1})]
\label{markov_crm}
\end{equation}
where $ \upol_{k}(x_k, p)=  \{q\in \Delta^L\mid   S^I_{k}(x_k, p)q\leq T^I_{k}(x_k, p), S^E_{k}(x_k, p)q= T^E_{k}(x_k, p)\}$ is the polytopic risk envelope, dependent only on the state $x_k$. We term such dynamic, time-consistent risk measures as \emph{Markov Dynamic Polytopic Risk Measures}.

\revision{Consistent with the stationarity assumption on the process disturbance pmf $p$, we will further assume that the polytopic risk envelopes $\upol_k$ are independent of time $k$ and state $x_k$, i.e. $\upol_k(x_k, p) = \upol(p)$, for all $ k$. We stress that our results readily generalize to the more general case as per Remark~\ref{rem:TV_p}.} 

%

\section{Problem Formulation}\label{sec:IHC}

In light of Sections \ref{sec:sys} and \ref{sec:risk}, we are now in a position to state the risk-sensitive optimization problem we wish to solve in this paper. We \revision{start by introducing} a notion of stability tailored to our risk-sensitive context.
\begin{definition}[Uniform Global Risk-Sensitive Exponential Stabilty]\label{stoch_stab_defn_exp}
System \eqref{eqn_sys} is said to be Uniformly Globally Risk-Sensitive Exponentially Stable (UGRSES) if there exist constants $c\geq0$ and $\lambda\in[0,1)$ such that for all initial conditions $x_0\in\reals^{N_x}$,
\begin{equation}
\risk_{0,k}(0,\ldots,0,x_k^\top x_k)\leq c\, \lambda^k \, x_0^\top x_0,\quad \emph{for all } k\in\naturals,\label{stab_ineq}
\end{equation}
where $\{\risk_{0,k}\}$ is a Markov dynamic polytopic risk measure. If condition \eqref{stab_ineq} only holds for initial conditions within some bounded neighborhood $\Omega$ of the origin, the system is said to be Uniformly Locally Risk-Sensitive Exponentially Stable (ULRSES) with domain $\Omega$.
\end{definition}
Note that, in general, UGRSES is a \emph{more restrictive} stability condition than mean-square stability\iftoggle{EV}{, as illustrated by the following example:
\begin{example}[Mean-Square Stability versus Risk-Sensitive Stability]
System \eqref{eqn_sys} is said to be Uniformly Globally Mean-Square Exponentially Stable (UGMSES) if there exist constants $c\geq0$ and $\lambda\in[0,1)$ such that for all initial conditions $x_0\in\reals^{N_x}$,
\begin{equation*}
\expectation{x_k^\top x_k} \leq c\, \lambda^k \, x_0^\top x_0,\quad \emph{for all } k\in\naturals,
\end{equation*}
see \cite[Definition 1]{RyashkoSchurz1996} and \cite[Definition 1]{BernardiniBemporad2012}. Consider the discrete time system
\begin{equation}\label{eq:sys_ex_mses}
x_{k+1}=\begin{cases}
\sqrt{0.5} \, x_k &\text{with probability } 0.2,\\
\sqrt{1.1} \, x_k  &\text{with probability } 0.8.\\
\end{cases}
\end{equation}
A sufficient condition for system \eqref{eq:sys_ex_mses} to be UGMSES is that there exist positive definite matrices $P = P^\top \succ 0$ and $L=L^\top \succ 0$ such that 
\[
\expectation{x_{k+1}^\top   P x_{k+1}} - x_k^\top  Px_k \leq -x_k^\top  L x_k,
\]
for all $k \in \naturals$, see \cite[Lemma 1]{BernardiniBemporad2012}. One can easily check that with $P=100$ and $L=1$ the above inequality is satisfied, and, hence system \eqref{eq:sys_ex_mses} is UGMSES.

Assuming risk is assessed according to the Markov dynamic polytopic risk measure $\risk_{0,k} = CVaR_{0.5}\circ\ldots\circ CVaR_{0.5}$, we next show that system \eqref{eq:sys_ex_mses} is \emph{not} UGRSES. In fact, using the dual representation of CVaR, one can write
\[ \small
\text{CVaR}_{0.5}(Z(x_{k+1}))=\max_{q\in\upol} \mathbb E_q[Z(x_{k+1})], 
\]
where $\upol =\left\{q\in \Delta^2  \mid 0\leq q_1 \leq 0.4,\,\, 0\leq q_2 \leq 1.6 \right\}$. Consider the pmf $\overline q=[0.1/1.1, \, 1/1.1]^\top $. Since $\overline  q \in \upol$, one has
\[
\text{CVaR}_{0.5}(x_{k+1}^2)\geq 0.5\, x_k^2\, \frac{0.1}{1.1} +  1.1\, x_k^2\, \frac{1}{1.1} = 1.0455 \, x_k^2. 
\]
By repeating this argument, one can then show that 
 \[
\risk_{0,k}(x_{k+1}^2) =  \text{CVaR}_{0.5}\circ\ldots\circ\text{CVaR}_{0.5}(x_{k+1}^2)\geq a^{k+1} \, x_0^\top x_0,
\]
where $a = 1.0455$. Hence, one cannot find constants $c$ and $\lambda$ that satisfy equation \eqref{stab_ineq}. Consequently, system \eqref{eq:sys_ex_mses} is UGMSES but \emph{not} UGRSES.\end{example}}{; in~\cite{SinghChowEtAl2018}, we provide an example to illustrate the difference between these two notions of stability.
}

Consider the MDP described in Section  \ref{sec:sys} and let $\Pi$ be the set of all stationary feedback control policies, i.e., $\Pi := \Bigl \{ \pi: \reals^{N_x} \rightarrow \reals^{N_u}\}$. Consider the quadratic cost function $C:\reals^{N_x}\times \reals^{N_u}\rightarrow \reals_{\geq 0}$ defined as 
$C(x,u):=\|x\|_Q^2+\|u\|_R^2$, where $Q=Q^\top \succ 0$ and $R=R^\top \succ 0$ are given state and control penalties, and $\|x\|_A^2$ defines the weighted norm, i.e., $x^{T}Ax$. Define the multi-stage cost function:
\[
J_{0,k}(x_0,\pi):=\risk_{0,k}\Bigl(C(x_0,\pi(x_0)),\ldots,C(x_{k},\pi(x_{k})) \Bigr),
\]
and the compact, convex constraint sets:
\[
	\begin{split}
	\mathbb{X} &:= \{ x \in \reals^{N_x} : \|T_x x \|_2 \leq x_{\max}\}, \\ 
	\mathbb{U} &:= \{ u \in \reals^{N_u} : \|T_u u \|_2 \leq u_{\max} \}.
	\end{split}
\]
The problem we wish to address is as follows.

\begin{quote} {\bf Optimization Problem $\mathcal{OPT}$} --- Given an initial state $x_0\in \reals^{N_x}$, solve
\begin{align*}
\inf_{\pi\in\Pi} \quad &\limsup_{k\rightarrow\infty}J_{0,k}(x_0,\pi) \\
\text{s.t.} \quad & x_{k+1}=A(w_k)x_{k}+B(w_k)\pi(x_k)\\
\quad &x_k \in \mathbb{X}, \pi(x_k) \in \mathbb{U} \ \forall k\\
\quad &\text{System is UGRSES}.
\end{align*}
\end{quote}

We denote the optimal cost function as $J^{\ast}_{0,\infty}(x_0)$. Note that the risk measure in the definition of  UGRSES is assumed to be identical to the risk measure used to evaluate the  cost of a policy.  Also, by time-invariance of the conditional risk measures, one can write
\begin{equation}
\risk_{0,k}\Bigl(C(x_0,\pi(x_0)),\ldots,C(x_{k},\pi(x_{k})) \Bigr) =C(x_0, \pi(x_0)) + \risk(C(x_1, \pi(x_1)) +\ldots+\risk(C(x_k, \pi(x_k)))\ldots), \label{eq:mar_pol_inv}
\end{equation}
where $\risk(\cdot)$ is a  given Markov polytopic risk measure that models the \revision{``degree" of risk sensitivity}.  This paper addresses problem $\mathcal{OPT}$ along three main dimensions: (1) find sufficient conditions for \emph{risk-sensitive} stability (i.e., for UGRSES); (2) design a convex MPC algorithm to efficiently compute a suboptimal state-feedback control policy; and (3) \revision {assess algorithm performance via numerical experiments}.

\section{Risk-Sensitive Stability}\label{sec:stab}

In this section we provide a sufficient condition for system \eqref{eqn_sys} to be UGRSES, under the assumptions of Section \ref{sec:IHC}. This condition relies on Lyapunov techniques and \revision{\emph{extends} to the risk-sensitive setting the condition provided} in \cite{BernardiniBemporad2012} (specifically, Lemma \ref{lyap_stab} reduces to Lemma 1 in \cite{BernardiniBemporad2012} when the risk measure is simply an expectation). 
\begin{lemma}[Sufficient Condition for UGRSES]\label{lyap_stab}
Consider a policy $\pi \in \Pi$ and the corresponding closed-loop dynamics for system \eqref{eqn_sys}, denoted by $x_{k+1}=f(x_k,w_k)$. The closed-loop system is UGRSES if there exists a function $V(x): \reals^{N_x}\rightarrow\reals$ and scalars $b_1,b_2,b_3>0$, such that for all $x\in\reals^{N_x}$,
\begin{equation}
\begin{split}
&b_1\, \|x\|^2\leq V(x)\leq b_2\|x\|^2, \,\, \text{and}\\
&\risk(V(f(x, w)))-V(x)\leq -b_3\|x\|^2.\label{eqn_RSES}
\end{split}
\end{equation}
\revision{We refer to the} function $V(x)$ as a \emph{risk-sensitive Lyapunov function}.
\end{lemma}
\iftoggle{EV}{
\begin{proof}
From the time-consistency, monotonicity, translational invariance, and positive homogeneity of Markov dynamic polytopic risk measures, condition \eqref{eqn_RSES} implies
\[
\begin{split}
 &\risk_{0,k+1}(0,\ldots,0,b_1\|x_{k+1}\|^2) \\
&\leq\risk_{0,k+1}(0,\ldots,0, V(x_{k+1})) \\
				&=\risk_{0,k}(0,\ldots,0,V(x_k)+\risk( V(x_{k+1})-V(x_k)))\\
 &\leq\risk_{0,k}(0,\ldots,0,V(x_k)-b_3\|x_k\|^2) \\
& \leq\risk_{0,k}(0,\ldots,0,(b_2-b_3)\|x_k\|^2).
 \end{split}
 \]
 Also, since $\risk_{0,k+1}$ is monotonic, one has $b_1\risk_{0,k+1}(0,\ldots,0,\|x_{k+1}\|^2)\geq 0$, which implies $b_2\geq b_3$ and in turn $(1-b_3/b_2)\in[0,1)$. 
Since $V(x_k)/b_2\leq \|x_k\|^2$, by using the previous inequalities one can write:
\begin{equation*} 
\begin{split}
\risk_{0,k+1}(0,\ldots,0, V(x_{k+1})) &\leq \risk_{0,k}(0,\ldots,0,V(x_k)-b_3\|x_k\|^2) \\
	&\leq \left(1\!-\!\frac{b_3}{b_2}\right)\risk_{0,k}\left(0,\ldots,0,V(x_k)\right).
\end{split}
\end{equation*}
Repeating this bounding process, one obtains:
\begin{equation*}
\begin{split}
&\risk_{0,k+1}(0,\ldots,0,V(x_{k+1})) \\
& \leq\left(1-\frac{b_3}{b_2}\right)^{k}\risk_{0,1}\left(V(x_1)\right) =\left(1-\frac{b_3}{b_2}\right)^{k} \risk \left(V(x_1)\right) \\
& \leq\left(1-\frac{b_3}{b_2}\right)^{k} \left(V(x_0) - b_3\|x_0\|^2 \right)\leq\, b_2\left(1-\frac{b_3}{b_2}\right)^{k+1} \, \|x_0\|^2.
\end{split}
\end{equation*}
Again, by monotonicity, the above result implies
\[
\risk_{0,k+1}(0,\ldots,0,x_{k+1}^\top x_{k+1})\leq\frac{b_2}{ b_1}\left(1-\frac{b_3}{b_2}\right)^{k+1}x_0^\top x_0.
\]
By setting $c=b_2/ b_1$ and $\lambda=(1-b_3/b_2)\in[0,1)$,  the claim is proven.
\end{proof}}{}
The closed-loop system is ULRSES with domain $\Omega$ if \eqref{eqn_RSES} holds within the bounded set $\Omega$.

\section{Model Predictive Control Problem}\label{sec:MPC}

This section describes a MPC strategy that approximates the solution to $\mathcal{OPT}$. \revision{We note that while an exact solution to $\mathcal{OPT}$ would lead to time-consistent risk assessments, MPC is not guaranteed to be time consistent over an infinite horizon realization, due to its {receding horizon nature}.} \revision{In this regard}, the MPC strategy  provides an efficiently implementable policy that approximately mimics the time-consistent nature of the optimal solution to $\mathcal{OPT}$. 

Our receding horizon framework \revision{consists of} two steps. First, offline, we search for the \emph{largest} ellipsoidal set $\emax$ and accompanying \emph{local} feedback control law $u(x) = Fx$ that renders $\emax$ control invariant and ensures satisfaction of \revision{the} state and control constraints. Additionally, within the offline step, we search for a terminal cost matrix $P$ (for the online MPC problem) to ensure that the closed-loop dynamics under the model predictive controller are risk-sensitive exponentially stable. The online MPC optimization then constitutes the second step of our framework. 

Consider first, the offline step. We parameterize $\emax$ as follows:
\begin{equation}
\emax (W):=\{ x\in \reals^{N_x} \mid \, x^{\top} W^{-1}x \leq 1\},
\label{term_set}
\end{equation}
where $W$ (and hence $W^{-1}$) is a positive definite matrix. The (offline) optimization problem to \revision{compute} $W$, $F$, and $P$ is presented below. 

\allowdisplaybreaks
\begin{quote} {\bf Optimization Problem $\mathcal{PE}$} --- Solve\vspace{-0.5cm}

{
\begin{align}
\hspace{-3mm} \max_{\substack{W = W^\top \succ 0  \\ P= P^\top \succ 0 \\ F}} \quad  &\text{logdet}(W)  \\
\text{s.t.} \quad & F^\top \frac{T_{u}^\top T_u}{u_{\max}^2}F-W^{-1}\preceq 0\label{control_LMI_1_stoch} \\
\quad &\begin{aligned}
&\sum_{j=1}^L q_{l}(j)\, (A_{j}+B_{j}F)^\top P(A_j+B_{j}F)  - P \\
&  +(F^\top RF+Q)\prec 0,\forall q_l\in\upolv \end{aligned} \label{term_ineq}\\
&\forall j\in\{1,\ldots,L\}:  \nonumber \\
& (A_{j}+B_{j}F)^\top \frac{T_{x}^\top T_x}{x_{\max}^2}(A_{j}+B_{j}F)-W^{-1}\preceq 0  \label{state_cons_LMI_1_stoch}\\
& (A_{j}+B_{j}F)^\top W^{-1}(A_j+B_{j}F)-W^{-1}\preceq 0.\label{state_term_cond_stoch} 
\end{align}
}%
\end{quote}
Note that inequality \eqref{term_ineq} (crucial to guaranteeing closed-loop risk-sensitive stability) is \emph{bi-linear} in the decision variables. In Section \ref{sec:alg}, we will derive an equivalent Linear Matrix Inequality (LMI) characterization of \eqref{term_ineq} in order to derive efficient solution algorithms. We first state the implications of problem $\mathcal{PE}$.

\begin{lemma}[Properties of $\mathcal{E}_{\max}$]\label{ctrl_inv}
Suppose problem $\mathcal{PE}$ is feasible and $x \in \mathbb{X} \cap \mathcal{E}_{\max}(W)$. Let $u(x) = Fx$. Then, the following statements are true:
\begin{enumerate}
	\item $\|T_{u}u\|_2\leq u_{\max}$, i.e., the control constraint is satisfied.
	\item $\|T_{x}\left(A(w)x + B(w)u\right)\|_2\leq x_{\max}$ surely, i.e., the state constraint is satisfied at the next step \revision{almost surely}.
	\item $A(w)x + B(w)u\in\mathcal{E}_{\max}(W)$ surely, i.e., the set $\mathcal{E}_{\max}(W)$ is robust control invariant under the control law $u(x)$. 
\end{enumerate}
Thus, $u(x) \in \mathbb{U}$ and $A(w)x + B(w)u \in \mathbb{X} \cap \mathcal{E}_{\max}(W)$ \revision{almost surely}.
\end{lemma}
\iftoggle{EV}{\begin{proof} See Appendix~\ref{app:mpc_aux}. \end{proof}}{}
Lemma \ref{ctrl_inv} establishes $\mathbb{X} \cap \emax(W)$ as a robust control invariant set under the feasible feedback control law $u(x) = Fx$. This result will be crucial to ascertain the persistent feasibility and closed-loop stability of the online optimization algorithm.

\iftoggle{EV}{
\begin{remark}
One can extend problem $\mathcal{PE}$ to the case with time-varying disturbance pmf by leveraging the state-space augmentation proposed in Remark~\ref{rem:TV_p}. Specifically, one would now require the existence of a \emph{state-dependent} feedback law for the terminal set $\mathcal{E}_{\max}(W)$. That is, instead of the terminal control law $u(x) = Fx$ for $x \in \mathcal{E}_{\max}(W)$, one now must search for a set of gain matrices $\{F_j\}_{j=1}^{L}$ where $u = F_j x $ is used when the augmented state is $(x, w^{[j]})$. Thus, this would simply induce an additional $L$ copies of the constraints in problem $\mathcal{PE}$.
\end{remark}
}{}

We are now ready to formalize the MPC problem. Suppose the feasible set of solutions in problem $\mathcal{PE}$ is non-empty and define $W=W^\ast$ and $P=P^\ast$, where $W^\ast,P^\ast$ are the maximizers for problem $\mathcal{PE}$. Given a prediction horizon $N \geq 1$, define the MPC cost function:
\begin{equation}\label{cost_MPC}
\begin{split}
J_k(x_{k|k},\ \pi_{k|k}, \ldots,\pi_{k+N-1|k},P) := \risk_{k,k+N}\big(&C(x_{k|k},\pi_{k|k}(x_{k|k})),\ldots, \\
&C(x_{k+N-1|k},\pi_{k+N-1|k}(x_{k+N-1|k})), \\
&C_P(x_{k+N|k})\big),
\end{split}
\end{equation}
where $x_{h|k}$ is the state at time $h$ predicted at stage $k$, $\pi_{h|k}:\mathbb{X} \rightarrow \mathbb{U}$ is the control \emph{policy} to be applied at time $h$ as determined at stage $k$, and $C_P(x) := x^T P x$ is the terminal cost function. Then, the \emph{online} MPC problem is formalized as:
\begin{quote} {\bf Optimization problem $\mathcal{MPC}$} --- Given current state $x_{k|k}\in \mathbb{X}$ and prediction horizon $N\geq 1$, solve \vspace{-0.5cm}

{
\begin{align}
\min_{ \substack{\pi_{k+h|k} \\ h \in [0,N-1]}  } \quad &J_k\left(x_{k|k},\pi_{k|k},\ldots,\pi_{k+N-1|k},P\right)  \\
\text{s.t.} \quad &  \begin{aligned} 
&x_{k+h+1|k}= A(w_{k+h})x_{k+h|k} +B(w_{k+h})\pi_{k+h|k}(x_{k+h|k})
\end{aligned}\\
\quad & \pi_{k+h|k}(x_{k+h|k}) \in \mathbb{U},\   x_{k+h+1|k} \in \mathbb{X}, \quad  h\in\{0,\ldots,N-1\} \label{mpc_sc_con}\\
\quad & x_{k+N|k}\in\mathcal E_{\max}(W)\,\,\text{surely}. \label{mpc_term_con}
\end{align}
}%
\end{quote}  
Note that a Markov policy is guaranteed to be optimal for problem $\mathcal{MPC}$ (see \cite[Theorem 2]{Ruszczynski2010}). The optimal cost function for problem $\mathcal{MPC}$ is denoted by $J^{*}_k(x_{k|k})$, and a minimizing policy is denoted by $\{\pi^\ast_{k|k}, \ldots, \pi^\ast_{k+N-1|k}\}$. For each state $x_k$, we set $x_{k|k} = x_k$ and the (time-invariant) model predictive control law is then defined as 
\begin{equation}\label{MPC_law_2} 
 \pi^{MPC}(x_k)=\pi^\ast_{k|k}(x_{k|k}).
\end{equation}

Note that problem $\mathcal{MPC}$ involves an optimization over \emph{time-varying closed-loop policies}, as opposed to the classical deterministic case where the optimization is over open-loop control inputs. We will show in Section \ref{sec:alg} how to solve problem $\mathcal{MPC}$ efficiently. We now address the persistent feasibility and stability properties for problem $\mathcal{MPC}$. 

\begin{theorem}[Persistent Feasibility]\label{feas_MPC}
Define $\mathcal{X}_N$ to be the set of initial states for which problem $\mathcal{MPC}$ is feasible.  
Assume $x_{k|k}\in\mathcal{X}_N$ and the control law is given by \eqref{MPC_law_2}. Then, it follows that $x_{k+1|k+1}\in\mathcal{X}_N$ surely. 
\end{theorem}
\iftoggle{EV}{\begin{proof} See Appendix~\ref{app:mpc_aux}. \end{proof}}{}
\begin{theorem}[Stochastic Stability with MPC]\label{stoch_stab_MPC}
Suppose the initial state $x_{0}$ lies within $\mathcal{X}_N$. Then, under the model predictive control law given in \eqref{MPC_law_2}, the closed-loop system is ULRSES with domain $\mathcal{X}_N$. 
\end{theorem}
\iftoggle{EV}{\begin{proof} See Appendix~\ref{app:stab_MPC}. \end{proof}}{}

\section{Solution Algorithms}\label{sec:alg}

Prior to solving problem $\mathcal{MPC}$, one would first need to find a matrix $P$ that satisfies \eqref{term_ineq} -- a bilinear semi-definite inequality in $(P,F)$. While checking feasibility of a bilinear semi-definite inequality is NP-hard \cite{TokerOzbay1995}, one can transform this inequality into an LMI by applying the Projection Lemma \cite{SkeltonIwasakiEtAl1997}. The next two results present LMI characterizations of conditions \eqref{control_LMI_1_stoch}--\eqref{state_term_cond_stoch}. \iftoggle{EV}{The proofs are provided in Appendix~\ref{app:soln_proj_lem}.}{}

\begin{theorem}[LMI Characterization of Stability Constraint]\label{thm_stab_1}
Let $\overline{A}:=\begin{bmatrix} {A}_1^\top &\ldots&A_{L}^\top \end{bmatrix}^\top$, $\overline{B}:=\begin{bmatrix} B^\top _{1}&\ldots&B_{L}^\top 
\end{bmatrix}^\top$, and for each $q_l\in\upolv$, define $\Sigma_{l}:=\mathrm{diag}(q_l(1),\ldots,q_l(L))\succ 0$. Consider the following set of LMIs with decision variables $Y$, $G$, $\overline Q=\overline Q^\top \succ 0$:
\begin{equation} \label{ineq_stab_1}
\begin{bmatrix}
I_{L\times L}\otimes\overline{Q}&0&0&-\Sigma^{\frac{1}{2}}_{l}(\overline{A}G+\overline{B}Y)\\
\ast&R^{-1}&0&-Y\\
\ast&\ast&I&-Q^{\frac{1}{2}}G\\
\ast&\ast&\ast&-\overline Q+G+G^\top \\
\end{bmatrix}\succ 0,\,
\end{equation}
for all $l\in\{1,\ldots,\mathrm{card}(\upolv)\}$. The expression in \eqref{term_ineq} is equivalent to the set of LMIs in \eqref{ineq_stab_1} by setting $F=YG^{-1}$ and $P=\overline Q^{-1}$.
\end{theorem}
Furthermore, by the application of the Projection Lemma to the expressions in (\ref{control_LMI_1_stoch}), (\ref{state_cons_LMI_1_stoch}) and (\ref{state_term_cond_stoch}), we obtain the following corollary:
\begin{corollary}\label{coro_inv}
Suppose the following set of LMIs with decision variables $Y$, $G$, and $W=W^\top \succ 0$ are satisfied:
\begin{equation} \label{ineq_cons_1}
\begin{split}
\begin{bmatrix}
x_{\max}^2I&-T_x({A}_jG+{B}_jY)\\
\ast&-W+G+G^\top \\
\end{bmatrix}\succ 0, \\
\begin{bmatrix} 
u_{\max}^2I&-T_uY \\
\ast&-W+G+G^\top \\
\end{bmatrix}\succ 0,\\ 
\begin{bmatrix} 
W&-({A}_jG+{B}_jY)\\
\ast&-W+G+G^\top \\
\end{bmatrix}\succ 0.
\end{split}
\end{equation}
Then, by setting $F=YG^{-1}$, the LMIs above represent \emph{sufficient} conditions for the LMIs in (\ref{control_LMI_1_stoch}), (\ref{state_cons_LMI_1_stoch}) and (\ref{state_term_cond_stoch}).
\end{corollary}
Note that in Corollary~\ref{coro_inv}, strict inequalities are imposed only for the sake of analytical simplicity when applying the Projection Lemma. Using similar arguments as in \cite{KothareBalakrishnanEtAl1996}, non-strict versions of the above LMIs may also be derived, for example, leveraging some additional technicalities \cite{Scherer1995}.

A solution approach for the receding horizon adaptation of problem $\mathcal{OPT}$ is to first solve the LMIs in Theorem \ref{thm_stab_1} and Corollary \ref{coro_inv}. If a solution for $(P,Y,G,W)$ is found, problem $\mathcal{MPC}$ can be solved via dynamic programming (see \cite[Theorem 2]{Ruszczynski2010}) after state and action \emph{discretization}, see, e.g., \cite{ChowPavone2013b,ChowTsitsiklis1991}. Note that the discretization process might yield a large-scale dynamic programming problem for which the computational complexity scales exponentially with the resolution of \revision{the} discretization. This motivates the convex programing approach presented next.

\subsection{Convex Programming Approach}

While problem $\mathcal{MPC}$ is defined as an optimization over \emph{Markov} control policies, in the convex programming approach, we re-define the problem as an optimization over \emph{history-dependent} policies. One can show (with a virtually identical proof) that the stability Theorem \ref{stoch_stab_MPC} still holds when history-dependent policies are considered. Furthermore, since Markov policies are optimal in our setup, the value of the optimal cost stays the same. The key advantage of history-dependent policies is that their additional flexibility leads to a convex formulation of the online problem. 
Consider the following parameterization of \emph{history-dependent} control policies. Let  $j_0,\ldots,j_{h}\in\{1,\ldots,L\}$ be the realized indices for the disturbances in the first $h+1$ steps of the $\mathcal{MPC}$ problem, where $h\in \{1, \ldots, N-1\}$. The control to be exerted at stage $h$ is denoted by ${U}_h(j_0,\ldots,j_{h-1})$. Similarly, we refer to the state at stage $h$ as ${X}_h(j_0,\ldots,j_{h-1})$. The dependence on $(j_0,\ldots,j_{h-1})$ enables us to keep track of the growth of the scenario tree. In terms of this new notation, the system dynamics \eqref{eqn_sys} can be rewritten as: \vspace{-0.5cm}

{
\begin{align}
{X}_1(j_0) = &A_{j_0}{X}_{0}+ B_{j_0}{U}_{0},\,\, &h=1, \nonumber\\
{X}_h(j_0,\ldots,j_{h-1}) = &A_{j_{h-1}}{X}_{h-1}(j_0,\ldots,j_{h-2})+B_{j_{h-1}}{U}_{h-1}(j_0,\ldots,j_{h-2}),\,\, &h\geq 2, \label{stoch_trans_ALG}
\end{align}
}%
where $X_0 := x_{k|k}$, and constraints~\eqref{mpc_sc_con} and~\eqref{mpc_term_con} \revision{can be rewritten} as:\vspace{-0.5cm}

{
\begin{align}
 {U}_0 \in \mathbb{U}, U_h(j_0,\ldots,j_{h-1}) &\in \mathbb{U},\ h\in \{1,\ldots,N-1\} \label{mpc_con} \\
 X_h(j_0,\ldots,j_{h-1}) &\in \mathbb{X},\ h\in \{1,\ldots,N\} \label{mpc_st} \\
 X_N(j_0,\ldots,j_{N-1}) &\in \mathcal{E}_{\max}, \label{mpc_term}
\end{align}
}%
for all $j_0,\ldots,j_{N-1} \in \{1,\ldots,L\}$. The final solution algorithm, termed convex MPC ($\mathcal{CMPC}$), is presented below. 
\begin{quote}{\bf Algorithm $\mathcal{CMPC}$} --- Given an initial state $x_{0} \in \mathbb{X}$ and a prediction horizon $N\geq 1$, solve
\[ \textbf{Offline} \] \vspace{-8mm}
{
\begin{align*}
	\max_{W=W^\top \succ 0,G, Y, \overline{Q}=\overline{Q}^\top \succ 0} \quad &\text{logdet}(W) \\
	\text{s.t.} \quad & \text{LMIs \eqref{ineq_stab_1} and (\ref{ineq_cons_1})}. 
\end{align*}
}
Denote optimizers: $\{W^\ast, \overline{Q}^\ast\}$.
\[ \textbf{Online} \]
\begin{enumerate}
\item Let $(W,P)=(W^\ast, (\overline{Q}^\ast)^{-1})$. At each step $k\in\{0,1,\ldots\}$, solve:
{
\begin{align*}
\hspace{-3mm} \min_{\footnotesize \substack{ U_0, {U}_{h}(\cdot), \\ h\in\{1,\ldots,N\} } } &\risk_{k,k+N}( C(x_{k|k},{U}_{0}),\ldots,C({X}_{N-1},{U}_{N-1}), C_P(X_N)) \\
\text{s.t.} \quad & \text{eqs. \eqref{mpc_con}--~\eqref{mpc_term}}.
\end{align*}
}
\item Set $\pi^{MPC}(x_{k|k}) = \overline{U}_{0}$.
\end{enumerate}
\end{quote}

\revision{The cost function in the online problem can be expressed as a nested sequence of convex quadratic inequalities by iteratively applying an epigraph reformulation (\iftoggle{EV}{see Appendix~\ref{app:mpc_cost_epi}}{see~\cite{SinghChowEtAl2018}} for an illustrative example). This results in a convex quadratically-constrained quadratic program (QCQP) which may be solved very efficiently even for moderate (clarified in the next section) values of $N$. In particular, the epigraph reformulation introduces an extra $O(L^{N-1})$ variables and $O(ML^{N-1})$ quadratic inequalities (where $M = \mathrm{card}(\upolv)$) to the existing $O(N_u L^{N-1})$ control variables in the online MPC problem.}

\iftoggle{EV}{
As a degenerate case, when we exclude all lookahead steps, problem $\mathcal{MPC}$ is reduced to an \emph{offline} optimization. By trading off performance, one can compute the control policy offline and implement it directly online without further optimization:

\begin{quote}{\bf Algorithm $\mathcal{MPC}^0$} --- Given $x_{0} \in \mathbb{X}$, solve:
\begin{align*}
\min_{\footnotesize \begin{array}{c}
\gamma_2,W=W^\top \succ 0, G, Y, \overline Q=\overline Q^\top \succ 0\\
\end{array}} \quad & \gamma_2 \\
\text{s.t.} \quad & \text{LMIs} \eqref{ineq_stab_1}, (\ref{ineq_cons_1}) \\
\quad & \begin{bmatrix}
1&x_0^\top \\
 \ast&W
\end{bmatrix}\succeq 0\ , \quad 
\begin{bmatrix}
 {\gamma}_{2}I&x_{0}^\top \\
 \ast&\overline Q
\end{bmatrix}\succeq 0.
\end{align*}
Then, set $\pi^{MPC}(x_k) = YG^{-1}x_k$.
\end{quote}
The domain of feasibility for $\mathcal{MPC}^0$ is the control invariant set $\mathbb{X}\cap \emax(W)$. Showing ULRSES for algorithm $\mathcal{MPC}^0$ is more straightforward than the corresponding analysis for problem $\mathcal{MPC}$ and is summarized within the following corollary. 

\begin{corollary}[Quadratic Lyapunov Function]\label{quad_lyap_coro} 
Suppose problem $\mathcal{MPC}^0$ is feasible. Then, system \eqref{eqn_sys} under the offline MPC policy: $\pi^{MPC}(x_k) = YG^{-1}x_k$ is ULRSES with domain $\mathbb{X}\cap \emax(W)$.
\end{corollary}
\begin{proof}
From Theorem \ref{thm_stab_1}, we know that the set of LMIs in \eqref{ineq_stab_1} is equivalent to the expression in \eqref{term_ineq} when $F = YG^{-1}$. Then since $x_0 \in \mathbb{X}\cap\emax(W)$, a robust control invariant set under the local feedback control law $u(x) = YG^{-1}x$, exploiting the dual representation of Markov polytopic risk measures yields the inequality
\begin{equation}
\risk_k(x_{k+1}^\top P x_{k+1})-x_k^\top Px_k\leq -x_k^\top Lx_k \ \  \forall k\in\naturals, \label{lyap_quad_example}
\end{equation}
where $L = Q + \left(YG^{-1}\right)^TR\left(YG^{-1}\right) = L^\top \succ 0$. Define the Lyapunov function $V(x)=x^\top Px$. Set $b_1=\lambda_{\min}(P)>0$, $b_2= \lambda_{\max}(P)>0$ and $b_3=\lambda_{\min}(L)>0$. Then by Lemma \ref{lyap_stab}, this stochastic system is ULRSES with domain $\mathbb{X}\cap \emax(W)$.
\end{proof} 
}{}

Note that our algorithms require a vertex representation of the polytopic risk envelopes (rather then the hyperplane representation in eq.~\eqref{eq:rep_fin}). In our implementation, we use the vertex enumeration function included in the MPT toolbox \cite{HercegKvasnicaEtAl2013}, which relies on the simplex method.

\section{Numerical Experiments}
\label{sec:example}

\revision{In this section} we present several numerical experiments that were run on a 2.6 GHz Intel Core i7 laptop, using the MATLAB YALMIP Toolbox (version 3.0 \cite{Loefberg2004}) with the Mosek solver~\cite{ApS2017}.

\subsection{Effects of Risk Aversion} \label{sec:ex_1}

In the first example, we consider the system studied in~\cite{BernardiniBemporad2012}, a similarly motivated work with an identical dynamical model but restricted to a risk neutral formulation. We remove the conditional dependence in the Markov chain governing $w_k$, i.e., all rows of the transition matrix governing the $w_k$ Markov chain are set to be the same to be consistent with model~\eqref{eqn_sys}. As per Remark~\ref{rem:TV_p}, the extension to the general case is straightforward. The goal of the first experiment is to study the effects of using a risk-sensitive objective.

\revisionII{Specifically,} consider \revisionII{the} second-order system defined by the transition matrices:
\[
	A_j = \begin{bmatrix} -0.8 & 1 \\ 0 & \bar{w}_j \end{bmatrix} , B_j = \begin{bmatrix} 0 & 1 \end{bmatrix}^T, \quad j \in \{1,2,3\},
\]
where $\bar{w}_j \in \{0.8, 1.2, -0.4\}$ with pmf $p = [0.5, 0.3, 0.2]$. The state and control constraints are defined by $T_x = \mathrm{diag}(1/10,1/2)$, $x_{\max} = 1$, $T_u = 1$, $u_{\max} = 1$. The cost matrices are $Q = \mathrm{diag}(1,5)$ and $R = 1$. We choose the conditional Markov polytopic risk measure: $\mathrm{CVaR}_{\alpha}$. Note that $\alpha=1$ corresponds to the standard risk neutral objective while $\alpha \ll 1$ corresponds to a worst-case risk assessment. For each value of $\alpha$ within the set $\{0.001, 0.5, 1.0\}$, we ran $1000$ simulations starting at $x_0 = (6,1)^T$ (a point lying in $\mathbb{X}\setminus \emax$), with 15 online MPC iteration steps and lookahead horizon $N=4$. Each MPC iteration took on average 24.82ms (for the case $\alpha = 0.5$ corresponding with the largest vertex set $\upolv$). In Figure~\ref{fig:ex_1_cost}, we plot the empirical cumulative density functions (cdfs) for the \emph{cumulative}\footnote{Recall that by translational invariance, $\risk_{0,k-1}(C_0,\ldots,C_k) = \risk_0\circ \cdots \circ \risk_{k-1}(C_0 +\cdots+C_k)$.} cost distribution at various time indexes.
\begin{figure}[h]
\centering
	\begin{subfigure}[b]{0.45\textwidth}
		\centering
		\includegraphics[width=\textwidth,clip]{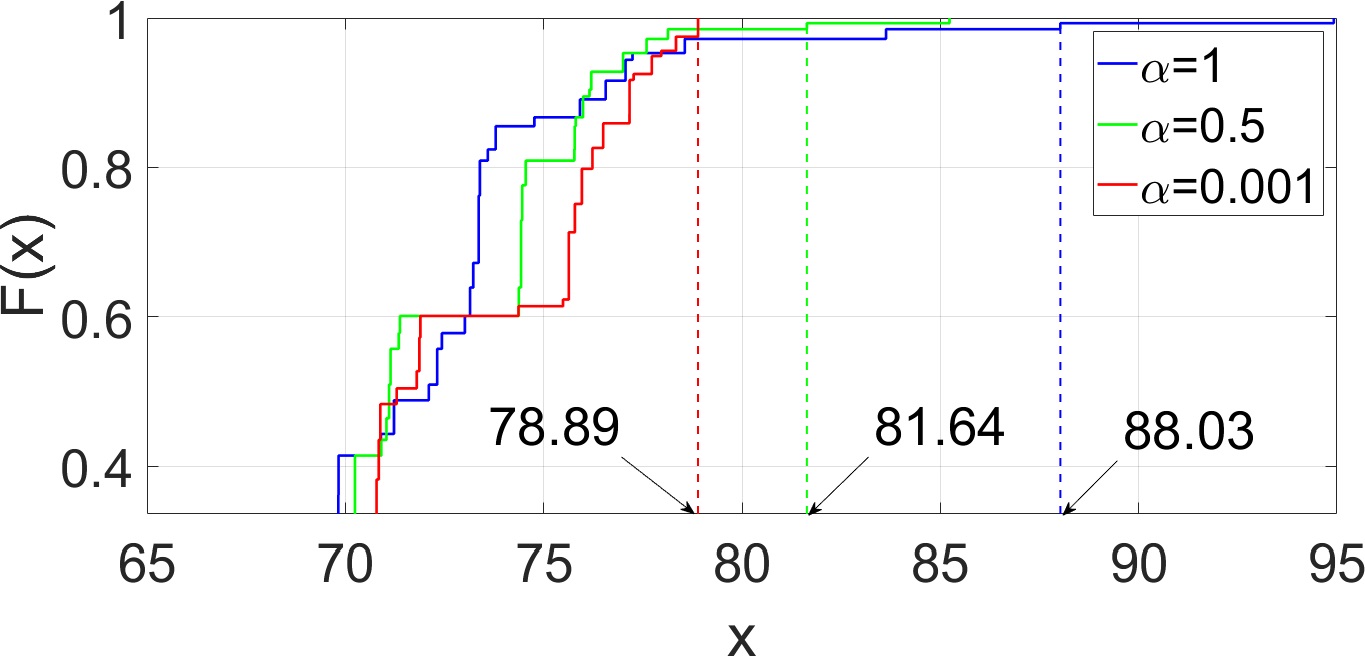}
		\caption{$k=3$}
		\label{fig:ex_1_k3}
	\end{subfigure}
	\begin{subfigure}[b]{0.45\textwidth}
		\centering
		\includegraphics[width=\textwidth,clip]{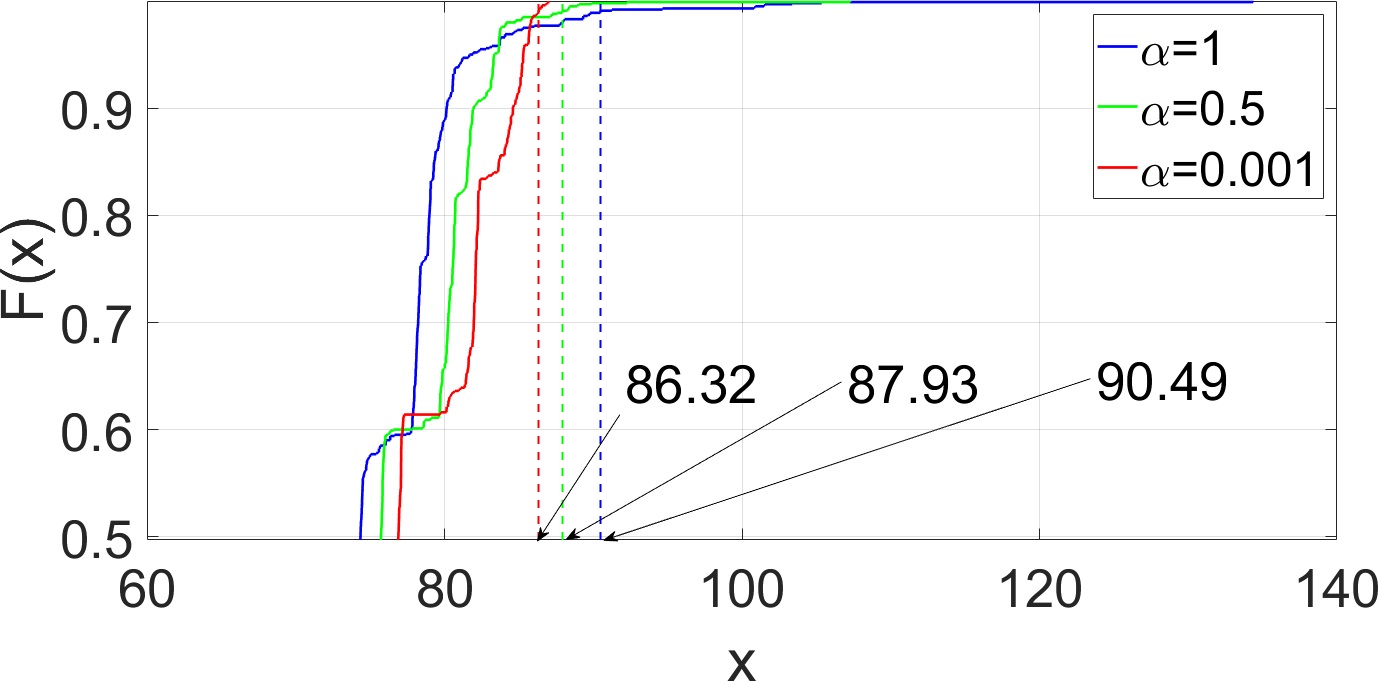}
		\caption{$k=7$}
		\label{fig:ex_1_k7}
	\end{subfigure}
	\begin{subfigure}[b]{0.45\textwidth}
		\centering
		\includegraphics[width=\textwidth,clip]{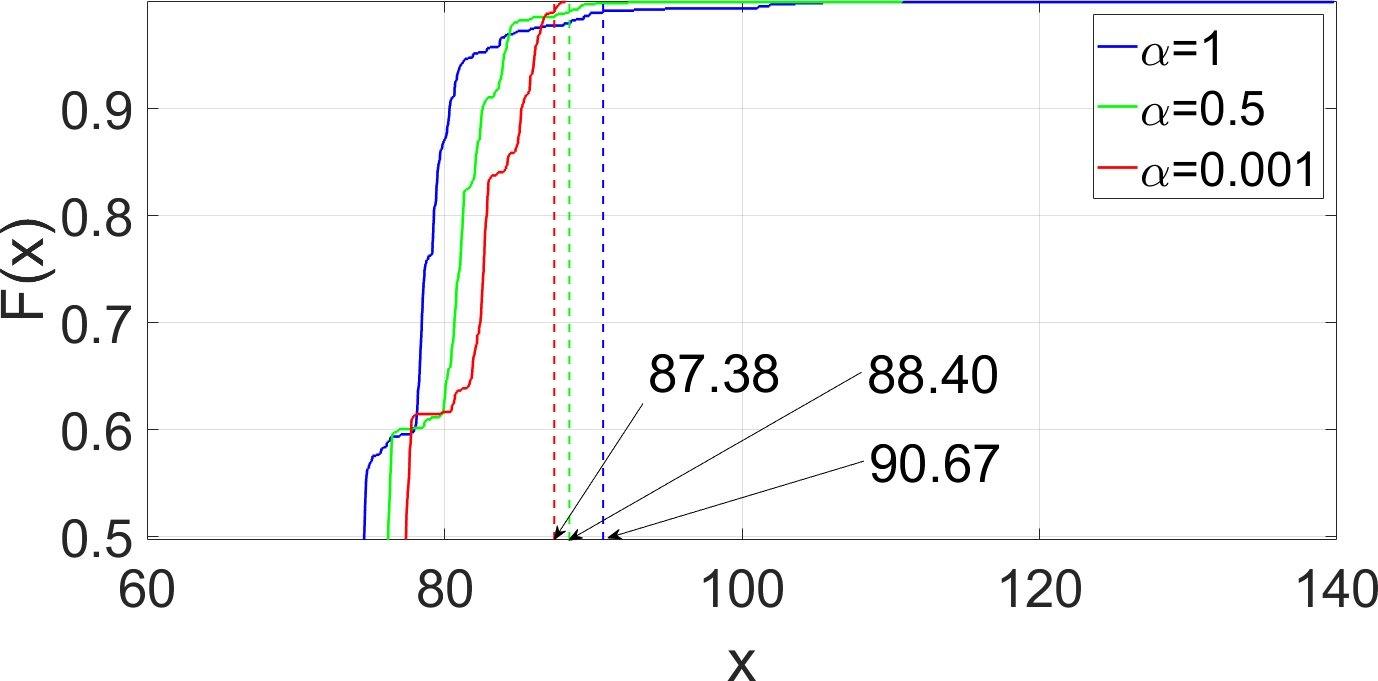}
		\caption{$k=11$}
		\label{fig:ex_1_k11}
	\end{subfigure}
	\begin{subfigure}[b]{0.45\textwidth}
		\centering
		\includegraphics[width=\textwidth,clip]{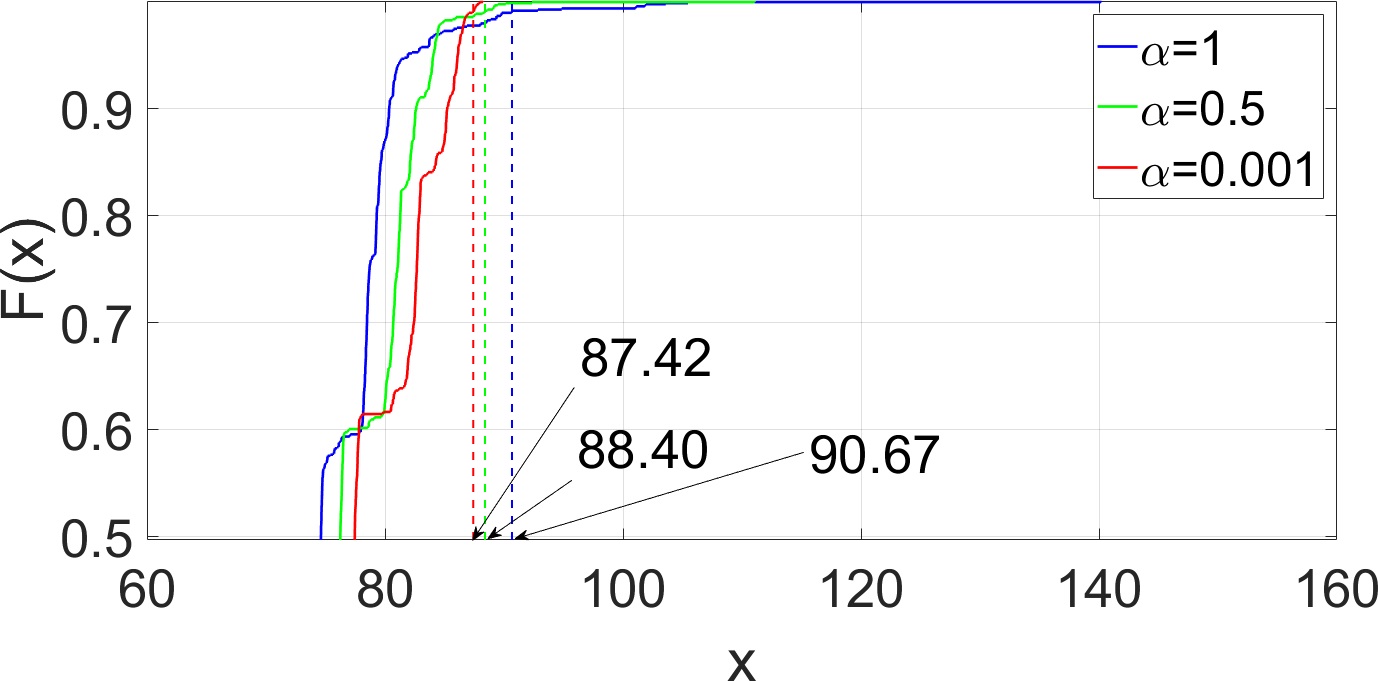}
		\caption{$k=14$}
		\label{fig:ex_1_k14}
	\end{subfigure}
	\caption{Cumulative cost cdfs for $k \in \{3, 7, 11, 14\}$. Dashed lines indicate the 0.99 quantile. The views presented are zoomed-in near the tail to emphasize the effects of risk-aversion. MPC horizon $N = 4$.}
\label{fig:ex_1_cost}
\end{figure}

Notice that as $\alpha$ decreases, the optimization further targets the high-cost tail of the cost distribution, at the expense of higher mean cost\footnote{One could additionally consider the convex combination $(1-\beta) \mathbb{E}[\cdot] + \beta\, \mathrm{CVaR}_{\alpha}(\cdot)$ for $\beta\in [0,1]$ to characterize the risk-sensitive Pareto trade-off curve.}. This is clearly observed in Figure~\ref{fig:ex_1_cost} that shows the tail quantile value decreasing as $\alpha$ decreases. Thus, using a single algorithm ($\mathcal{CMPC}$), we are able to generate tunable risk-sensitive policies from risk-neutral to worst-case.

\subsection{Computational Limits of $\mathcal{CMPC}$} \label{sec:ex_2}
In this example, we randomly generate a set of $L=6$ systems using MATLAB's $\mathrm{drss}$ function with $N_x = 5$ and $N_u = 2$ to investigate the \emph{computational} limits of our algorithms. The constraints are defined by $(T_x, x_{\max}) = (2 I_5, 5)$, $(T_u, u_{\max}) = (I_2,1)$, and the cost is defined by the weighting matrices $Q = 2I_5$ and $R = I_2$. The nominal pmf $p$ is randomly generated and experiments were performed with conditional risk measure $\mathrm{CVaR}_{\alpha}$ with $\alpha = 0.2$ (yielding a vertex set of size 20), varying lookahead horizons, and 15 MPC iterations for each simulation. The results are summarized  in Table~\ref{tab:comp}.

\begin{table}[h]
\centering
       \begin{tabular}{c|c|c}
         \hline
     $N$ & \# Scenarios & Mean (max) [s] \\
     \hline
      2 & 7 & 0.02 (0.035) \\
      3 & 43 & 0.15 (0.211)\\
      4 & 259 & 1.84 (2.612) \\
      5 & 1555 & 54.73 (60.89)\\
          \hline
  \end{tabular}
   \caption{Solve times per MPC iteration for varying lookahead horizons $N$ for a system with $N_x=5$, $N_u = 2$, $L=6$, over 100 simulations. The \# Scenarios column corresponds to the number of control nodes in each online MPC scenario tree.} \label{tab:comp}
\end{table}

The table illustrates the applicability of the algorithm on a fairly large (with respect to the number of problem variables and constraints) dimensional example, with appreciable lookahead. The exponential growth in computation time is an unavoidable feature of scenario-based optimization. Thus, as typical of jump dynamic systems, we envision the applicability of this work to model temporally-extended (i.e., mode-switching) dynamics as opposed to fast dynamical systems. Improving the runtime capabilities of this algorithm must undoubtedly rely on massively parallel sampling in concert with branch-and-bound techniques, and is left for future research. 

\section{Conclusion and Future Work}\label{sec:conclusion}

In this paper we presented a framework for risk-sensitive MPC by leveraging recent advances in the theory of dynamic risk measures developed by the operations research community. The proposed approach has the following advantages: (1) it is axiomatically justified and leads to time-consistent risk assessments; (2) it is amenable to dynamic and convex programming; and (3) it is general, in that it captures a full range of risk assessments from risk-neutral to worst case (due to the generality of Markov polytopic risk measures). Our framework thus provides a unifying perspective on risk-sensitive MPC.

This paper opens several directions for future research. First, we plan to extend our work to handle cases where the state and control constraints are required to hold only with a given probability threshold (in contrast to hard constraints) by exploiting techniques such as \emph{probabilistic invariance}~\cite{CannonKouvaritakisEtAl2009}. This relaxation has the potential to provide significantly improved performance at the risk of occasionally violating constraints. Second, we plan to combine our approach with methods for scenario tree optimization in order to reduce the online computation load. Third, while polytopic risk measures encompass a wide range of possible risk assessments, extending our work to non-polytopic risk measures and more general stage-wise costs can broaden the domain of application of the approach. Fourth, we plan to generalize this framework to allow for nonlinear dynamics and more expressive models of uncertainty (e.g., time-varying distributions). Fifth, an important consideration from a practical standpoint is the \emph{choice} of risk measure appropriate for a given application. We plan to develop principled approaches for making this choice, e.g., by computing polytopic risk envelopes based on confidence regions for the disturbance model.



\bibliographystyle{IEEEtran} 
\newcommand{\noopsort}[1]{} \newcommand{\printfirst}[2]{#1}
  \newcommand{\singleletter}[1]{#1} \newcommand{\switchargs}[2]{#2#1}


\appendices

\section{Proof of Lemma \ref{ctrl_inv} and Theorem \ref{feas_MPC}}\label{app:mpc_aux}

 \begin{proof}[Proof of Lemma~\ref{ctrl_inv}]
We first prove the first and second statements and thereby establish $u(x)$ as a feasible control law within the set $\emax(W)$. Notice that:
\begin{equation}
	\|T_{u}Fx\|_2\leq u_{\max} \Leftrightarrow \|T_{u}FW^{\frac{1}{2}}(W^{-\frac{1}{2}} x)\|_2\leq u_{\max}.
\label{eq:feas_lem}
\end{equation}
From \eqref{term_set}, applying the Schur complement, we know that $\|W^{-\frac{1}{2}}x\|_2\leq 1$ for any $x\in\mathcal{E}_{\max}(W)$. Thus, by the Cauchy Schwarz inequality, a sufficient condition for~\eqref{eq:feas_lem} is given by $\|T_{u}FW^{\frac{1}{2}}\|_2\leq u_{\max}$, which can be written as
\[
(FW^{\frac{1}{2}})^\top T_{u}^\top T_{u}(FW^{\frac{1}{2}})\preceq u_{\max}^2 I \Leftrightarrow F^\top T_{u}^\top T_{u}F\preceq u_{\max}^2W^{-1}.
\]
Re-arranging the inequality above yields the expression given in \eqref{control_LMI_1_stoch}. The state constraint can be proved in an identical fashion by leveraging \eqref{term_set} and \eqref{state_cons_LMI_1_stoch}. It is omitted for brevity.

We now prove the third statement. By definition of a robust control invariant set, we are required to show that for any $x\in\mathcal{E}_{\max}(W)$, that is, for all $x$ satisfying the inequality: $x^\top W^{-1}x\leq 1$, application of the control law $u(x)$ yields the following inequality:
\[
(A_{j}x + B_{j}Fx)^\top W^{-1}(A_{j}x+ B_{j}Fx)\leq 1,\forall j\in\{1,\ldots,L\}.
\]
Equivalently, by the S-procedure \cite{Yakubovich1977}, we are required to show the existence of a $\lambda\geq 0$ such that the following condition holds:
\[\small
\begin{bmatrix}
\lambda W^{-1}-(A_{j} + B_{j}F)^\top W^{-1}(A_{j} + B_{j}F)&0\\
\ast& 1-\lambda
\end{bmatrix}\succeq 0, 
\]
for all $j\in\{1,\ldots,L\}$. By setting $\lambda=1$, one obtains the largest feasibility set for $W$ and $F$. The expression in \eqref{state_term_cond_stoch} corresponds to the (1,1) block in the matrix above. 
 \end{proof}
 
\begin{proof}[Proof of Theorem~\ref{feas_MPC}]
Given $x_{k|k}\in\mathcal{X}_{N}$, problem $\mathcal{MPC}$ may be solved to yield a closed-loop optimal control policy:
\[
\{\pi^*_{k|k}(x_{k|k}),\ldots,\pi^*_{k+N-1|k}(x_{k+N-1|k})\}, 
\]
such that $x_{k+N|k} \in \mathbb{X} \cap \mathcal{E}_{\max}(W)$ almost surely. Consider problem $\mathcal{MPC}$ at stage $k+1$ with initial condition $x_{k+1|k+1}$. From Lemma \ref{ctrl_inv}, we know that 
\begin{equation}\label{policy_k1}
\{\pi^*_{k+1|k}(x_{k+1|k}),\ldots,\pi^*_{k+N-1|k}(x_{k+N-1|k}),Fx_{k+N|k}\}, 
\end{equation}
is a feasible control policy at stage $k+1$. Note that this is simply a concatenation of the optimal tail policy from the previous iteration $\{\pi^{*}_{k+h|k}(x_{k+h|k})\}_{h=1}^{N-1}$, with the state feedback law $Fx_{k+N|k}$ for the final step. 

Since a feasible control policy exists at stage $k+1$, $x_{k+1|k+1} = A_{j}x_{k|k}+B_{j}\pi^*_{k|k}(x_{k|k}) \in \mathcal{X}_N$ for any $j\in\{1,\ldots,L\}$, completing the proof. 
\end{proof}

\section{Closed-loop Stability of MPC} \label{app:stab_MPC}
\begin{proof}[Proof of Theorem~\ref{stoch_stab_MPC}]
Let $J^*_k(x_{k|k})$ denote the optimal value function for problem $\mathcal{MPC}$. We will show that $J^*_k$ is a risk-sensitive Lyapunov function (Lemma \ref{lyap_stab}). Specifically, we first show that $J^*_k$ satisfies the two inequalities in equation \eqref{eqn_RSES}. Consider the bottom inequality in equation \eqref{eqn_RSES}. At time $k$ consider problem $\mathcal{MPC}$ with state $x_{k|k}$. 
The sequence of optimal control policies is given by $\{\pi^*_{k+h|k} \}_{h=0}^{N-1}$. Now, consider the sequence:
\begin{equation}\label{policy_k1}
\pi_{k+h|k+1}(x_{k+h|k}) := \{\pi^*_{k+1|k}(x_{k+1|k}),\ldots,\pi^*_{k+N-1|k}(x_{k+N-1|k}),Fx_{k+N|k}\}, \nonumber 
\end{equation}
which, as we know from Lemma~\ref{ctrl_inv}, is a feasible solution to problem $\mathcal{MPC}$ at stage $k+1$. Thus, for problem $\mathcal{MPC}$ at stage $k+1$ with initial condition given by $x_{k+1|k+1}=A(w_k)x_{k|k}+B(w_k)\pi_{k|k}^\ast(x_{k|k})$, denote by $\overline{J}_{k+1}(x_{k+1|k+1})$, the problem $\mathcal{MPC}$ objective corresponding to the control policy sequence $\pi_{k+h|k+1}(x_{k+h|k})$. Note that  $x_{k+1|k+1}$ (and therefore $\overline{J}_{k+1}(x_{k+1|k+1})$) is a random variable with $L$ possible realizations, given $x_{k|k}$. Define:
\begin{equation*}
\begin{split}
Z_{k+N}:=\ &x_{k+N|k}^T\left(-P + Q + F^TRF\right) x_{k+N|k},\\
Z_{k+N+1}:= &\left( (A(w_{k+N|k}) + B(w_{k+N|k})F) x_{k+N|k} \right)^T  P \\
& \left( (A(w_{k+N|k}) + B(w_{k+N|k})F) x_{k+N|k} \right).
\end{split}
\end{equation*}
By exploiting the dual representation of Markov polytopic risk metrics, one can write
\begin{equation*}
\begin{split}
&Z_{k+N}+\risk_{k+N}(Z_{k+N+1}) = x_{k+N|k}^T\left(-P +Q + F^T R F  \right) x_{k+N|k}  \\
&+ \max_{q\in \upol(p)} \sum_{j=1}^L q(j) x_{k+N|k}^T \left(A_j + B_j F \right)^T  P \,  \left( A_j + B_j F \right) x_{k+N|k}.
\end{split}
\end{equation*}
Combining the equation above with equation \eqref{term_ineq}, one readily obtains the inequality 
\begin{equation}\label{stab_ineq_key}
Z_{k+N}+\risk_{k+N}(Z_{k+N+1})\leq 0.
\end{equation}
One can then construc the following chain of inequalities:
{\small 
\begin{align}
J^{*}_k(x_{k|k})&=C(x_{k|k},\pi^\ast_{k|k}(x_{k|k}))+ \risk_k\Biggl(\risk_{k+1, N}\Bigl(C(x_{k+1|k}, \pi^*_{k+1|k}(x_{k+1|k})), \ldots,  \|x_{k+N|k}\|^{2}_{Q}   +  \|x_{k+N|k}\|^{2}_{F^T R F} \ + \nonumber \\ 
& \hspace{7cm} \risk_{k+N}(Z_{k+N+1}) - Z_{k+N} \!- \!\risk_{k+N}({Z_{k+N+1}}) \!\Bigr) \!\!\Biggr)\! \! \nonumber \\
&\geq C(x_{k|k},\pi^\ast_{k|k}(x_{k|k}))+ \risk_k\Biggl(\risk_{k+1, N}\Bigl(C(x_{k+1|k}, \pi^*_{k+1|k}(x_{k+1|k})), \ldots, \|x_{k+N|k}\|^{2}_{Q}  \ +  \|x_{k+N|k}\|^{2}_{F^T R F} + \\
&\hspace{11cm} \risk_{k+N}(Z_{k+N+1}) \Bigr) \Biggr) \nonumber \\
&= C(x_{k|k},\pi^\ast_{k|k}(x_{k|k}))\!+\!\risk_k\Bigl({\overline{J}_{k+1}(x_{k+1|k+1})}\Bigr) \nonumber  \\
&\geq C(x_{k|k},\pi^\ast_{k|k}(x_{k|k}))+\risk_k\Bigl(J^*_{k+1}(x_{k+1|k+1})\Bigr), \label{eq:bottom}
\end{align}
}%
where the first equality follows from the definitions of $Z_{k+N}$ and of dynamic, time-consistent risk measures, the second inequality follows from equation \eqref{stab_ineq_key}  and the monotonicity property of Markov polytopic risk metrics (see also  \cite[Page 242]{Ruszczynski2010}), the third equality 
follows from the definition of $\overline{J}_{k+1}(x_{k+1|k+1})$, 
and the fourth inequality follows from the definition of $J^*_{k+1}$ and the monotonicity of Markov polytopic risk metrics. 

Consider now the top inequality in equation \eqref{eqn_RSES}. One can easily bound $J^{*}_k(x_{k|k})$ from below according to:
\begin{equation}\label{eq:low_con}
J^{*}_k(x_{k|k})\geq x_{k|k}^T Q x_{k|k}\geq \lambda_{\min}(Q)\|x_{k|k}\|^2,
\end{equation}
where $\lambda_{\min}(Q)>0$ by assumption. The upper bound for $J^{*}_k(x_{k|k})$ is derived in two steps. First, define
\[
\begin{split}
& M_A:=\max_{r\in\{0,\ldots,N-1\}}\max_{j_0,\ldots,j_{r}\in\{1,\ldots,L\}}\alpha_{j_{r}}\ldots \alpha_{j_{1}}\alpha_{j_0},\,\, \\ &\text{where }  \alpha_j:=\|A_j + B_jF\|_2.
\end{split}
\] 
Suppose $x_{k|k} \in \mathbb{X}\cap \emax(W)$. From Lemma \ref{ctrl_inv}, we know that the control policy $\pi_{k+h|k}(x_{k+h|k}) = \{Fx_{k+h|k}\}_{h=0}^{N-1}$ is feasible and consequently, $\mathbb{X}\cap \emax(W) \subseteq \mathcal{X}_N$. Defining $\theta_f:=\|Q +F^\top R F\|_2$, we thus have
\begin{equation*}
\begin{split}
J^{*}_k(x_{k|k})\leq &C\left(x_{k|k},Fx_{k|k}\right)+\risk_k\Bigl( C\left(x_{k+1|k},Fx_{k+1|k}\right)+ \ldots+\risk_{k+N-1}\left(x_{k+N|k}^\top P x_{k+N|k}\right)\ldots\Bigr)\\
\leq&\theta_f\|x_{k|k}\|_2^2 +\risk_k\Bigl( \theta_f\|x_{k+1|k}\|_2^2+\ldots+\risk_{k+N-1}\left( \|P\|_2\|x_{k+N}\|^2_2\right)\ldots\Bigr),
\end{split}
\end{equation*}
for all $x_{k|k} \in \mathbb{X}\cap \emax(W)$. Exploiting the translational invariance and monotonicity property of Markov polytopic risk metrics, one obtains the upper bound for all $x_{k|k} \in \mathbb{X}\cap \emax(W)$:
\begin{equation}
J^{*}_k(x_{k|k})\leq \underbrace{\left(N\, \theta_f+ \|P\|_2\right )M_A}_{:=\beta >0 }\|x_{k|k}\|_2^2. 
\end{equation}
In order to derive an upper bound for $J^{*}_k(x_{k|k})$ with the above structure for all $x_{k|k} \in \mathcal{X}_N$, we draw inspiration from a similar proof in \cite[Proposition 2.18]{RawlingsMayne2013}. By leveraging the finite cardinality of the disturbance set $\mathcal{W}$ and the set closure preservation property of the inverse of continuous functions, it is possible to show that $\mathcal{X}_N$ is closed. Then, since $\mathcal{X}_N$ is necessarily a subset of the bounded set $\mathbb{X}$, it follows that $\mathcal{X}_N$ is compact. Thus, there exists some constant $\Gamma>0$ such that $J^{*}_k(x_{k|k}) \leq \Gamma$ for all $x_{k|k}\in \mathcal{X}_N$. That $\Gamma$ is finite follows from the fact that $\{\|x_{k+h|k}\|_2\}_{h=0}^{N}$ and $\{\| \pi_{k+h|k}(x_{k+h|k}) \|_2 \}_{h=0}^{N-1}$ are finitely bounded for all $x_{k|k} \in \mathcal{X}_N$. Now since $\emax(W)$ is compact and non-empty, there exists a $d>0$ such that $\mathcal{E}_{d}:= \{x\in \reals^{N_x} \mid \|x\|_2 \leq d\} \subset \emax(W)$. Let $\hat{\beta} = \max\{\beta \|x\|_2^2 \mid \|x\|_2 \leq d \}$. 

Consider, now, the function: $(\Gamma/\hat{\beta}) \beta\|x\|_2^2$. Then since $\beta\|x\|^2_2 > \hat{\beta}$ for all $x \in \mathcal{X}_N \setminus \mathcal{E}_{d}$ and $\Gamma\geq\hat{\beta}$, it follows that
\begin{equation}\label{eq:up}
J^{*}_k(x_{k|k}) \leq \left(\dfrac{\Gamma \beta}{\hat{\beta}}\right)\|x_{k|k}\|_2^2,\,\, \forall x_{k|k} \in \mathcal{X}_N,
\end{equation}
as desired. Combining the results in equations \eqref{eq:bottom}, \eqref{eq:low_con}, \eqref{eq:up}, and given the time-invariance of our problem setup, one concludes that $J^{*}_k(x_{k|k})$ is a risk-sensitive Lyapunov function for the closed-loop system \eqref{eqn_sys}, in the sense of Lemma \ref{lyap_stab}. This concludes the proof.
\end{proof}

\section{Proof of Theorem \ref{thm_stab_1} and Corollary \ref{coro_inv}}\label{app:soln_proj_lem}
We first present the Projection Lemma:
\begin{lemma}[Projection Lemma]\label{proj_lem}
For matrices $\Omega(X)$, $U(X)$, $V(X)$ of appropriate dimensions, where $X$ is a matrix variable, the following statements are equivalent:
\begin{enumerate}
\item There exists a matrix $W$ such that
\[
\Omega(X)+U(X)WV(X)+V(X)^\top W^\top U(X)^\top \prec 0.
\]
\item The following inequalities hold:
\[
\begin{split}
&U(X)^{\perp}\Omega(X) (U(X)^{\perp}))^\top \prec 0,\,\,(V(X)^\top )^{\perp}\Omega(X) ((V(X)^\top )^{\perp})^\top \prec 0,
\end{split}
\]
\end{enumerate}
where $A^\perp$ is the orthogonal complement of $A$.
\end{lemma}
\begin{proof}
See Chapter 2 in \cite{SkeltonIwasakiEtAl1997}.
\end{proof}

We now give the proof for Theorem \ref{thm_stab_1} by leveraging the Projection lemma:

\begin{proof}(Proof of Theorem \ref{thm_stab_1})
Using simple algebraic factorizations, for all $l\in\{1,\ldots,\text{cardinality}\left(\csd^{\text{poly},V}(p)\right)\}$, inequality (\ref{term_ineq}) can be be rewritten as
\[
\begin{bmatrix}
I\\
\Sigma^{\frac{1}{2}}_{l}(\overline{A}+\overline{B}F)\\
F\\
Q^{\frac{1}{2}}\\
\end{bmatrix}^\top \!\!\begin{bmatrix}
P&0&0&0\\
0&-I_{L\times L}\otimes{P}&0&0\\
0&0&- R&0\\
0&0&0&- I\\
\end{bmatrix}\!\!\begin{bmatrix}
I\\
\Sigma^{\frac{1}{2}}_{l}(\overline{A}+\overline{B}F)\\
F\\
Q^{\frac{1}{2}}\\
\end{bmatrix}\succ 0.
\]
By Schur complement, the above expression is equivalent to
\begin{equation}
\begin{bmatrix}
I&0&0&\Sigma_{l}^{\frac{1}{2}}(\overline{A}+\overline{B}F)\\
0&I&0&F\\
0&0&I&Q^{\frac{1}{2}}\\
\end{bmatrix}\begin{bmatrix}
I_{L\times L}\otimes\overline{Q}&0&0&0\\
0&R^{-1}&0&0\\
0&0&I&0\\
0&0&0&-\overline Q\\
\end{bmatrix}\begin{bmatrix}
I&0&0\\
0&I&0\\
0&0&I\\
(\overline{A}+\overline{B}F)^\top\Sigma^{\frac{1}{2}}_{l} &F^\top &Q^{\frac{1}{2}}\\
\end{bmatrix}\succ 0,
\label{proj_1}
\end{equation}
where $\overline Q=P^{-1}$. Now since $\overline Q=\overline Q^\top \succ 0$ and $R=R^\top \succ 0$, we also have the following identity:
\begin{equation}
\begin{bmatrix}
I&0&0&0\\
0&I&0&0\\
0&0&I&0\\
\end{bmatrix}\begin{bmatrix}
I_{L\times L}\otimes\overline{Q}&0&0&0\\
0&R^{-1}&0&0\\
0&0&I&0\\
0&0&0&-\overline Q\\
\end{bmatrix}\begin{bmatrix}
I&0&0\\
0&I&0\\
0&0&I\\
0&0&0\\
\end{bmatrix}\succ 0.
\label{proj_2}
\end{equation}
Next, notice that
\[
\begin{bmatrix}
-\Sigma^{\frac{1}{2}}_{l}(\overline{A}+\overline{B}F)\\
-F\\
-Q^{\frac{1}{2}}\\
I\\
\end{bmatrix}^\perp = 
\begin{bmatrix}
I&0&0&\Sigma_{l}^{\frac{1}{2}}(\overline{A}+\overline{B}F)\\
0&I&0&F\\
0&0&I&Q^{\frac{1}{2}}\\
\end{bmatrix},\quad \begin{bmatrix}
0\\
0\\
0\\
I\\
\end{bmatrix}^\perp = 
\begin{bmatrix}
I&0&0&0\\
0&I&0&0\\
0&0&I&0
\end{bmatrix}.
\]
Now, set:
\[
\Omega = -\begin{bmatrix}
I_{L\times L}\otimes\overline{Q}&0&0&0\\
0&R^{-1}&0&0\\
0&0&I&0\\
0&0&0&-\overline Q\\
\end{bmatrix},\ U = \begin{bmatrix}
-\Sigma^{\frac{1}{2}}_{l}(\overline{A}+\overline{B}F)\\
-F\\
-Q^{\frac{1}{2}}\\
I\\
\end{bmatrix},\ V^T = \begin{bmatrix}
0\\
0\\
0\\
I\\
\end{bmatrix}.
\]
Then by Lemma \ref{proj_lem}, inequalities~\eqref{proj_1} and~\eqref{proj_2} are equivalent to the existence of a matrix $G$ that satisfies the following inequality for all $l\in\{1,\ldots,\text{cardinality}\left(\csd^{\text{poly},V}(p)\right)\}$:
\begin{equation}\label{proj_lem_use_stab}
\begin{bmatrix}
I_{L\times L}\otimes\overline{Q}&0&0&0\\
0&R^{-1}&0&0\\
0&0&I&0\\
0&0&0&-\overline Q\\
\end{bmatrix}+\begin{bmatrix}
-\Sigma^{\frac{1}{2}}_{l}(\overline{A}+\overline{B}F)\\
-F\\
-Q^{\frac{1}{2}}\\
I\\
\end{bmatrix}G\begin{bmatrix}
0\\
0\\
0\\
I\\
\end{bmatrix}^\top \\
+\begin{bmatrix}
0\\
0\\
0\\
I\\
\end{bmatrix}G^\top \begin{bmatrix}
-\Sigma^{\frac{1}{2}}_{l}(\overline{A}+\overline{B}F)\\
-F\\
-Q^{\frac{1}{2}}\\
I\\
\end{bmatrix}^\top \succ 0.
\end{equation}
Setting $F=YG^{-1}$ and pre-and post-multiplying the above inequality by $\text{diag}(I,R^{\frac{1}{2}},I, I)$ yields the LMI given in \eqref{ineq_stab_1}. Furthermore, from the inequality $-\overline Q+G+G^\top \succ 0$ where $\overline Q\succ 0$, we know that $G+G^\top \succ 0$. Thus, by the Lyapunov stability theorem, the linear time-invariant system $\dot{x} = -Gx$ with Lyapunov function $x^\top x$ is asymptotically stable (i.e., all eigenvalues of $G$ have positive real part). Therefore, $G$ is an invertible matrix and $F=YG^{-1}$ is well defined.
\end{proof}

\begin{proof}(Proof of Corollary \ref{coro_inv})
We will prove that the third inequality in (\ref{ineq_cons_1}) implies inequality (\ref{state_term_cond_stoch}). Details of the proofs on the implications of the first two inequalities in (\ref{ineq_cons_1}) follow from identical arguments and will be omitted fin the interest of brevity.
Using simple algebraic factorizations, inequality (\ref{state_term_cond_stoch}) may be rewritten (in strict form) as:
\[
\begin{bmatrix}
I\\
A_j+B_jF\\
\end{bmatrix}^\top \!\!\begin{bmatrix}
W^{-1}&0\\
0&-W^{-1}\\
\end{bmatrix}\!\!\begin{bmatrix}
I\\
A_j+B_jF\\
\end{bmatrix}\succ 0,\,\,\forall j\in\{1,\ldots,L\}.
\]
By Schur complement, the above expression is equivalent to
\begin{equation}
\begin{bmatrix}
I&A_j+B_jF\\
\end{bmatrix}\begin{bmatrix}
W&0\\
0&-W\\
\end{bmatrix}\begin{bmatrix}
I\\
(A_j+B_jF)^\top\\
\end{bmatrix}\succ 0,\,\,\forall j\in\{1,\ldots,L\}.
\label{proj_3}
\end{equation}
Furthermore since $W\succ 0$, we also have the identity
\begin{equation}
\begin{bmatrix}
I&0
\end{bmatrix}\begin{bmatrix}
W&0\\
0&-W\\
\end{bmatrix}\begin{bmatrix}
I\\
0\\
\end{bmatrix}\succ 0.
\label{proj_4}
\end{equation}
Now, notice that:
\[
\begin{bmatrix}
-({A}_j+{B}_jF)\\
I\\
\end{bmatrix}^\perp=
\begin{bmatrix}
I&{A}_j+{B}_jF\\
\end{bmatrix},\qquad \begin{bmatrix}
0\\
I\\
\end{bmatrix}^\perp=\begin{bmatrix}
I&0
\end{bmatrix}.
\]
Then by Lemma \ref{proj_lem}, inequalities~\eqref{proj_3} and~\eqref{proj_4} are equivalent to the existence of a matrix $G$ such that the following inequality holds  for all $j\in\{1,\ldots,L\}$:
\begin{equation}\label{proj_lem_use_inv}
\begin{bmatrix}
W&0\\
0&-W\\
\end{bmatrix}+\begin{bmatrix}
-({A}_j+{B}_jF)\\
I\\
\end{bmatrix}G\begin{bmatrix}
0\\
I\\
\end{bmatrix}^\top \\
+\begin{bmatrix}
0\\
I\\
\end{bmatrix}G^\top \begin{bmatrix}
-({A}_j+{B}_jF)\\
I\\
\end{bmatrix}^\top \succ 0.
\end{equation}
Note that Lemma \ref{proj_lem} provides an equivalence (necessary and sufficient) condition between (\ref{proj_lem_use_inv}) and (\ref{state_term_cond_stoch}) if $G$ is allowed to be any arbitrary LMI variable. However, in order to restrict $G$ to be the same variable as in Theorem \ref{thm_stab_1}, the equivalence relation reduces to sufficiency only. Setting $F = YG^{-1}$ in the above expression gives the claim.
\end{proof}

\section{Epigraph Reformulation of MPC Cost Function}\label{app:mpc_cost_epi}
Suppose $N=2$. We wish to solve
\[
	\min_{u_0, \pi_{1}} C(x_0,u_0) + \rho(C(x_1,\pi_1(x_1)) + \rho(C_T(x_2)))
\]
where $C_T$ is the terminal state cost $x_2^T P x_2$. Using the polytopic dual representation of a coherent risk measure and the history-dependent parameterization, the MPC problem can be equivalently written as:
\begin{align}
\min_{U_0,U_1(\cdot),\tau_0,\tau_1} \quad & C(X_0,U_0) + \tau_0 \\
\text{subject to} \quad & \tau_0 \geq q_l^T \left( C(X_1(\cdot),U_1(\cdot)) + \tau_1 \right) \quad \forall q_l \in \upolv \\
\quad &\tau_1(j_0) \geq q_l ^T C_T(X_2(j_0, \cdot))  \quad \forall j_0 = 1,\ldots,L,\ q_l \in \upolv \\
\quad & \text{State, Control, and terminal set constraints}.
\end{align}
This epigraph reformulation is justified by the fact that the maximum in~\eqref{markov_crm} must occur at one of the vertices in $\upolv$.

\end{document}